\documentclass[12pt]{amsart}
\usepackage{ifpdf}
\usepackage[colorlinks,linkcolor=blue,citecolor=blue,urlcolor=blue, pdfcenterwindow, pdfstartview={XYZ null null 1.2}, pdffitwindow, pdfdisplaydoctitle=true]{hyperref}
\usepackage[english]{babel}
\usepackage{array}
\usepackage{threeparttable}
\usepackage{rotating}
\usepackage[section]{placeins}
\usepackage{amssymb} 
\usepackage{amsmath}
\usepackage{mathrsfs}
\usepackage{breqn}
\usepackage{bbm}
\usepackage{enumerate}
\usepackage{color}
\usepackage{amsrefs} 

\begin{document}
\newtheorem{lemma}{Lemma}[section]
\newtheorem{lemm}[lemma]{Lemma}
\newtheorem{prop}[lemma]{Proposition}
\newtheorem{coro}[lemma]{Corollary}
\newtheorem{theo}[lemma]{Theorem}
\newtheorem{conj}[lemma]{Conjecture}
\newtheorem{prob}{Problem}
\newtheorem{ques}{Question}
\newtheorem{rema}[lemma]{Remark}
\newtheorem{rems}[lemma]{Remarks}
\newtheorem{defi}[lemma]{Definition}
\newtheorem{defis}[lemma]{Definitions}
\newtheorem{exam}[lemma]{Example}

\newcommand{\N}{\mathbf N}
\newcommand{\Z}{\mathbf Z}
\newcommand{\R}{\mathbf R}
\newcommand{\Q}{\mathbf Q}
\newcommand{\C}{\mathbf C}

\title{Ergodic boundary representations}

\author{A. Boyer}
\address{Weizmann Institute of Science, Rehovot, Israel}
\email{adrien.boyer@weizmann.ac.il}
\author{G. Link}
\address{KIT, Institut f\"ur Algebra und Geometrie}
\email{gabriele.link@kit.edu}
\author{Ch. Pittet}
\address{I2M, UMR 7373 CNRS, Aix-Marseille Universit\'e}
\email{pittet@math.cnrs.fr}



\thanks{The first author is supported by the ERC Grant 306706.  The authors thank the MFO and the CIRM for
RIP contracts providing excellent research environments}

\date{September 15th, 2016}

\maketitle
%
\begin{abstract} We prove a von Neumann type ergodic theorem for averages of unitary operators arising from the  Furstenberg-Poisson boundary representation (the quasi-regular representation) of any lattice in a non-compact connected semisimple Lie group with finite center. 
\end{abstract} 


\section{Introduction}
Let $\Gamma$ be an infinite  finitely generated group and let $\pi$ be a unitary representation of $\Gamma$ on a Hilbert space ${\mathcal H}$. What can be learned
from the asymptotic behavior of weighted averages of the type
\[
\sum_{\gamma\in\Gamma}a_{\gamma}\pi(\gamma)
\]
with $a_{\gamma}\in\mathbb C$ carefully chosen? This question has been intensively studied for representations associated to measurable actions. The weakest form of the von Neumann ergodic theorem (convergence in the weak operator topology) is one of the founding results of this line of thought. 
Recall that a sequence $A_n$ in the Banach  algebra ${\mathcal B}({\mathcal H})$ of bounded operators on ${\mathcal H}$ converges to $A\in {\mathcal B}({\mathcal H})$ with respect to  the  weak operator topology (WOT) 
if and only if for any $v,w\in {\mathcal H}$
\[
	\lim_{n\to\infty}\langle A_nv,w\rangle=\langle Av,w\rangle.
\]
\subsection{Unitary representations and measurable transformations}
Suppose $\Gamma$ acts by measure preserving transformations on a probability space $(B,\nu)$. Let $\pi_{\nu}$ be the associated canonical unitary representation on the Hilbert space $L^2(B,\nu)$. That is:
$$\pi_{\nu}(\gamma)\varphi(b)=\varphi(\gamma^{-1}b)\quad \forall b\in B\  \forall\varphi\in L^2(B,\nu)\  \forall \gamma\in\Gamma.$$
Let $\mathbbm 1_B$ be the characteristic function of the whole space $B$ and let $P_{\mathbbm 1_B}\in {\mathcal B}(L^2(B,\nu))$ denote the orthogonal projection 
onto the complex line generated by $\mathbbm 1_B$.
\emph{The existence of a sequence $\mu_n\in L^1(\Gamma)$  such that 
\[
	\lim_{n\to\infty}\pi_{\nu}(\mu_n)=\lim_{n\to\infty}\sum_{\gamma}\mu_n(\gamma)\pi_{\nu}(\gamma)=P_{\mathbbm 1_B}
\]
in the WOT implies the ergodicity of the action}.

\emph{If $\Gamma$ is amenable, the converse implication also holds with $\mu_n$ the uniform probability measures defined by any F{\o}lner sequence $F_n$: if the action is ergodic, then in the WOT 
\[
\lim_{n\to\infty}\frac{1}{|F_n|}\sum_{\gamma\in F_n}\pi_{\nu}(\gamma)=P_{\mathbbm 1_B}.
\]} This is a straightforward consequence of the $L^2$-mean ergodic theorem for amenable groups. More generally, \emph{for any locally compact second countable amenable group $G$, a measure preserving action on a probability space $(B,\nu)$ is ergodic if and only if
for any  F{\o}lner sequence $F_n$
\[
\lim_{n\to\infty}\frac{1}{\mathrm{vol}(F_n)}\int_{F_n}\pi_{\nu}(g){\rm d}g=P_{\mathbbm 1_B}
\]
in the WOT, where ${\rm d}g$ denotes a Haar measure on $G$ and $\mathrm{vol}(F_n)$ denotes the corresponding Haar volume of $F_n$}, see \cite[Ch. IV]{Tem} and references therein, see also Corollary 4.6 in F. Pogorzelski's Diploma thesis ``Ergodic Theorems on Amenable Groups'' (Eberhard-Karls-Universität Tübingen, 2010). See also \cite[5.1 p. 894]{Nev} for a short proof of the $L^2$-mean ergodic theorem which is a variant of Riesz's proof of von Neumann's ergodic theorem.
In fact \emph{the ergodicity of the action implies a.e. point-wise convergence of Birkhoff sums provided the F{\o}lner sequence is carefully chosen} \cite{Lin}.

Let us mention some results for measure preserving actions of non-amenable groups. \emph{The ergodicity of the action implies the  convergence in the WOT of $\pi_{\nu}(\mu_n)$ to $P_{\mathbbm 1_B}$}
in the following cases: the group \emph{$\Gamma$ is Gromov-hyperbolic and $\mu_n$ is the family of Cesaro means with respect to concentric balls  or spherical shells defined  by any word metric associated to a finite symmetric generating set of $\Gamma$} (see \cite{Gui} for the special case of the free groups), the group  \emph{$\Gamma$ is a lattice in a connected semisimple Lie group with finite center and $\mu_n$ is the uniform measure on the intersection of $\Gamma$ with a bi-$K$-invariant lift in $G$ of a ball of radius $n$ in the symmetric space $G/K$}. See \cite{BowNev},\cite{Nev},\cite{GorNev},\cite{Ana}  and references therein for many other and stronger convergence results ($L^p$-convergence, point-wise a.e. convergence, equidistribution, rates of convergence, etc.). 

When the acting group $\Gamma$ is non-amenable there may be no invariant measures but only quasi-invariant ones. In this case the associated quasi-regular representation is made  unitary with the help of the Radon-Nikodym cocycle:
\[
\pi_{\nu}(\gamma)\varphi(b)=\varphi(\gamma^{-1}b)\sqrt{\frac{{\rm d}\gamma_*\nu}{{\rm d}\nu}(b)}\quad\forall b\in B\  \forall\varphi\in L^2(B,\nu).	
\]
Due to the possible rapid decay of the Radon-Nikodym cocycle as $\gamma$ goes to infinity, if $E_n\subset\Gamma$ and $|E_n|\to\infty$, \emph{the averages 
\[
	\frac{1}{|E_n|}\sum_{\gamma\in E_n}\pi_{\nu}(\gamma)
\]
may converge in the WOT to the zero operator}. (For example, \emph{this is the case if $\Gamma$ is any uniform lattice in $SL(2,\mathbb R)$,
respectively in a non-compact connected semisimple Lie group $G$ with finite center, and $\pi_{\nu}$ is the quasi-regular representation of $\Gamma$ on the circle at infinity of $SL(2,\mathbb R)/SO(2,\mathbb R)$, respectively on the Furstenberg-Poisson boundary of $G$,  and $E_n$ is the ball in $\Gamma$ of radius $n$ defined by a word metric associated to any finite symmetric set of $\Gamma$}.)  Therefore it makes sense to normalize each unitary operator $\pi_{\nu}(\gamma)$ by the mean (the Harish-Chandra function
associated to $\pi_{\nu}$ evaluated on $\gamma$)
\[
	\Xi(\gamma)=\langle\pi_{\nu}(\gamma)\mathbbm 1_B,\mathbbm 1_B\rangle=\int_B\sqrt{\frac{{\rm d}\gamma_*\nu}{{\rm d}\nu}(b)}{\rm d}\nu(b)
\]
and rather considering the averages 
\[
	\frac{1}{|E_n|}\sum_{\gamma\in E_n}\frac{\pi_{\nu}(\gamma)}{\langle\pi_{\nu}(\gamma)\mathbbm 1_B,\mathbbm 1_B\rangle}.
\]
\subsection{Bader-Muchnik's theorem on negatively curved manifolds and generalizations}
Let $M$ be a closed Riemannian manifold with strictly negative sectional curvature and let $(X,d_X)$ be the universal cover of $M$ endowed with the unique Riemannian metric locally isometric to $M$. Fix a base point $x_0\in X$. For any $x\in X$ with  $x\neq x_0$ there is a unique geodesic ray $c_x:[0,\infty[\to X$ such that $c_x(0)=x_0$ and $c_x\bigl(d_X(x_0,x)\bigr)=x$. Let $B=\partial X$ be the visual boundary of $X$ endowed with the Patterson-Sullivan
measure $\nu$ corresponding to $x_0\in X$. Up to normalization, $\nu$ is the Hausdorff measure defined by the Bourdon metric $d_{x_0}$ on $B$ (see \cite{Bou}) and $\Gamma$ acts by conformal transformations on $B$ with Radon-Nikodym cocycle
\[
	\frac{{\rm d}\gamma_*\nu(b)}{{\rm d}\nu(b)}=e^{-\delta B_b(\gamma x_0,x_0)};
\]
here $\delta$ is the Hausdorff dimension of $(B,d_{x_0})$ and $B_b(\gamma x_0,x_0)$ is the Busemann function on $X$ defined by  the point $b\in B$ (see \cite{Rob}).
As $\Gamma=\pi_1(M)$ acts (from the left) freely on $X$, each $\gamma\in\Gamma\setminus\{e\}$ defines a unique geodesic ray $c_{\gamma x_0}$ emanating from $x_0$ and passing through $\gamma x_0$. Let $c_{\gamma x_0}(\infty)\in B$ be the point it defines in the visual boundary $B$ and  consider the  boundary map 
$$\bold{b}:\Gamma\setminus\{e\}\to B$$
defined by $\bold{b}(\gamma)=c_{\gamma x_0}(\infty).$ 
 The map  $\bold{b}$ allows to associate to any function 
$f$ defined on the boundary $B$ and any finite subset $E$ of $\Gamma\setminus\{e\}$ the weighted average
\[
	\frac{1}{|E|}\sum_{\gamma\in E}f(\bold{b}(\gamma))\frac{\pi_{\nu}(\gamma)}{\langle\pi_{\nu}(\gamma)\mathbbm 1_B,\mathbbm 1_B\rangle}.
\]	
If $f\in L^{\infty}(B,\nu)$ we denote  $m(f)\in {\mathcal B}(L^2(B,\nu))$ the multiplication operator defined by $f$, that is:  $m(f)\phi=f\phi$ $\ \forall \phi\in L^2(B,\nu)$. In a remarkable paper \cite{BadMuc}, U. Bader and R. Muchnik prove  that \emph{if $f$ is continuous then in the WOT
\[
\lim_{n\to\infty}\frac{1}{|B_n|}\sum_{\gamma\in B_n}f(\bold{b}(\gamma))\frac{\pi_{\nu}(\gamma)}{\langle\pi_{\nu}(\gamma)\mathbbm 1_B,\mathbbm 1_B\rangle}=m(f)P_{\mathbbm 1_B},	
\]
where $B_n$ is the ball of radius $n$ in $\Gamma$ defined by the length function $L(\gamma)=d_X(x_0,\gamma x_0)$ with center $e\in\Gamma$ removed}  (the statement in \cite{BadMuc} involves annuli rather than balls but this is equivalent; see Proposition \ref{prop: ergodicity relative to balls and annuli} below).

The starting point of the present work was the guess -- hinted by the authors in their paper: ``We will resist the temptation of stating things in a greater generality than needed"  -- that Bader-Muchnik's theorem and its consequences hold in a much more general setting than the one of closed Riemannian manifolds with strictly negative curvature.  
One of the authors of this paper (A. Boyer)
generalized Bader-\-Mu\-chnik's theorem to  discrete isometry groups of proper $\mathrm{CAT}(-1)$ spaces, having non-arithmetic length spectrum, finite Bowen-Margulis-Sul\-livan measure and $\delta$-Ahlfors regular Patterson-Sul\-li\-van conformal density of dimension $\delta$ (see \cite{BoyCAT}). This covers convex-cocompact groups with non-arithmetic spectrum as well as finite volume locally symmetric spaces of rank one, but does not cover complete finite volume Riemannian manifolds with pinched negative curvature (see the interesting examples from \cite{DalPeiPicSam}).
Recall that the length spectrum is non-arithmetic if by definition the subgroup of the real line
generated by the lengths of the closed geodesics of the quotient space is not cyclic, and that a compact locally $\mathrm{CAT}(-1)$ space has 
an arithmetic length spectrum if and only if it is a finite
graph with edge lengths rationally dependent \cite[Thm. 4]{Ricks}. In the case the $\mathrm{CAT}(-1)$ space is the Cayley graph (with all edges of length $1$) of the free group on $n$ letters, all closed geodesics in the wedge of $n$ circles of length $1$ have integral length, hence the spectrum is obviously arithmetic and the general theory does not apply. Nevertheless a direct counting argument \cite{BoyPin} allows to prove  Bader-Muchnik's theorem in this case as well. Following the same lines of ideas as Bader and Muchnik, L. Garncarek was able to prove that the boundary representation associated to a Patterson-Sullivan measure of a Gromov hyperbolic group is irreducible \cite{Gar}. In the same vein, the irreducibility of some boundary representations associated to Gibbs measures has been established in \cite{BoyMay}.

In the unpublished note ``A brief presentation of property RD"  2006, E. Breuillard speculates that a special case of Bader-Muchnik's theorem should also hold true for lattices in semisimple Lie groups of higher rank and that the proof should easily 
follow from the work of Gorodnik and Oh \cite{GorOh}. More precisely, Breuillard asks if the quasi-regular representation $\lambda_{G/P}$ of a connected  non-compact semisimple Lie group $G$ with  finite center  on its Furstenberg-Poisson boundary $(G/P,\nu)$ satisfies for any lattice $\Gamma\subset G$
\[
\lim_{T\to\infty}\frac{1}{|S_r(T)|}\sum_{\gamma\in S_r(T)}\frac{\langle\lambda_{G/P}(\gamma)\phi,\psi\rangle}{\langle\lambda_{G/P}(\gamma)\mathbbm 1_{G/P},\mathbbm 1_{G/P}\rangle}=\int_B\phi(b) {\rm d}\nu(b)\int_B\overline{\psi(b)}{\rm d}\nu(b)
\]
for all (positive) $\phi,\psi\in L^2(G/P,\nu)$, where $$S_r(T)=\{\gamma\in \Gamma : T-r\leq L(\gamma)\leq T+r \}$$ is the intersection with $\Gamma$ of the annulus in $G$ around  $e\in G$ of radius $T$ and thickness $r$ with respect to a length function $L$ on $G$.

\subsection{Statement of the main result}\label{subsection: statement of the main results}

Let $G$ be a connected semisimple Lie group with finite center and let $\frak g$ be its Lie algebra.
Let $K$ be a maximal compact subgroup of $G$ and let $\frak k$ be its Lie algebra. Let $\frak p$ be the orthogonal complement of $\frak k$ in $\frak g$
relative to the Killing form $B$. Among the  abelian sub-algebras of $\frak g$ contained in the subspace $\frak p$, let $\frak a$ be a maximal one. We assume $\dim \frak a>0$, i.e. the real rank of $G$ is strictly positive (or equivalently that $G$ is not compact). Let $\Sigma$ be the root system associated to $(\frak g,\frak a)$. Let
\[
	\frak g_{\alpha}=\{X\in\frak g: \mathrm{ad}(H)X=\alpha(H)X\ \ \forall H\in\frak a  \}
\]
be the root space of $\alpha\in\Sigma$.
Let 
\[
	\frak a^{sing}=\bigcup_{\alpha\in\Sigma}\mathrm{ker}(\alpha)
\]
be the union of the hyperplanes of $\frak a$ defined as the kernels of the roots of $\Sigma$. 
Let us choose a positive open  Weyl chamber $\frak a^+$ (i.e. a connected component of $\frak a\setminus \frak a^{sing}$).
Let $\Sigma^+\subset\Sigma$ be the set of positive roots ($\alpha\in\Sigma$ is positive if and only if $\alpha(H)>0$ for all $H\in \frak a^+$), and 
 $\frak n$ the nilpotent Lie algebra defined as the direct sum of root spaces of positive roots:
\[
	\frak n=\bigoplus_{\alpha\in\Sigma^+}\frak g_{\alpha}.
\]
Let $A=\exp(\frak a)$, $A^+=\exp(\frak a^+)$ and $N=\exp(\frak n)$.
Let $G=KAN$ be the Iwasawa decomposition defined by $\frak a^+$. Let $Z(A)$ be the centralizer of $A$ in $G$ and  
 $M=Z(A)\cap K$. The  group $M$ normalizes $N$. Let $P=MAN$ be the minimal parabolic subgroup of $G$ associated to $\frak a^+$.
Let $\nu$ be the unique Borel regular $K$-invariant probability measure on the Furstenberg-Poisson boundary $G/P$. 
Let 
\[
\lambda_{G/P}:G\to \mathcal{U}(L^2(G/P,\nu))	
\]
be the associated quasi-regular representation and let 
\[
	\Xi(g)=\langle\lambda_{G/P}(g)\mathbbm 1_{G/P},\mathbbm 1_{G/P}\rangle
\]
be the Harish-Chandra function (see Subsections \ref{subsection: The quasi-regular representation and the Harish-Chandra function associated to a quasiinvariant measure} and \ref{subsection: Sjogren} below for definitions and references).
The subset $KA^+K\subset G$ is open and dense, and the map
\[
	\bold{b}:KA^+K\to G/P
\]
defined as 
\[
	\bold{b}(k_1ak_2)= k_1P
\]
is continuous (to check that $\bold{b}$ is well defined see \cite[Theorem 5.20 and its proof]{Kna} or \cite[Theorem 1.1 and Corollary 1.2 of Ch. IX and their proofs]{HelDifGeo}).
Let $d_G$ be the left-invariant Riemannian metric on $G$ whose scalar product on the tangent space of $G$ at the identity $e$
\[
	\langle Z,Y\rangle=-B(Z,\Theta(Y))\quad \forall Z,Y\in\frak g
\]
is defined by the Killing form $B$ and the Cartan involution $\Theta$ associated to the decomposition
\[
	\frak g=\frak k\oplus\frak p.
\]
It is convenient to consider also normalizations of the scalar product $\langle \cdot,\cdot \rangle$ and $d_G$. Once a normalization is chosen,
let $d_X$ be the unique Riemannian metric on the symmetric space $X=G/K$
such that the canonical projection $G\to G/K$ becomes a $G$-equivariant Riemannian submersion.
Let $x_0\in X$ be the image of the identity element of $G$ under the canonical projection.
Notice that if $H\in\frak a$, then
\[
	d_X(\exp(H)x_0,x_0)=\|H\|.
\]
Let $T>0$. We define
\[
	\frak a_T=\{H\in\frak a: \|H\|< T\},\quad \frak a^+_T=\frak a^+\cap\frak a_T.
\]
We denote the images in $G$ under the exponential map as
\[
	A_T=\exp(\frak a_T),\quad A^+_T=\exp(\frak a^+_T). 
\]
Eventually, we define 
\[
	B_T=KA^+_TK.
\]
If $\psi\in L^{\infty}(G/P,\nu)$, we denote $m(\psi)\in {\mathcal B}\bigl(L^{2}(G/P,\nu)\bigr)$ the corresponding multiplication operator.
Let $P_{\mathbbm 1_{G/P}}\in {\mathcal B}\bigl(L^{2}(G/P,\nu)\bigr)$ be the orthogonal projection onto the complex line spanned by $\mathbbm 1_{G/P}$.
Let $\Gamma$ be a discrete subgroup of $G$,
\[
	\Gamma_T=B_T\cap\Gamma
\]
and let $|\Gamma_T|$ denote the cardinality of this finite set. 
Let $f:G/P\to\mathbb C$ be a continuous function. In the case  $\Gamma_T$ is non-empty we may consider 
the bounded operator
\[
	M_{\Gamma_T}^f=\frac{1}{|\Gamma_T|}\sum_{\gamma\in \Gamma_T}f(\bold {b}(\gamma))\frac{\lambda_{G/P}(\gamma)}{\Xi(\gamma)}.
\]  

\begin{theo}\label{theo: ergodicity of the quasi-regular representation of a lattice in a semisimple Lie group}(Ergodicity of the quasi-regular representation of a lattice in a semisimple Lie group.) Let $G$ be a non-compact connected semisimple Lie group with finite center. Let $P$ be a minimal parabolic subgroup of $G$ and $\Gamma$ a lattice in $G$. Let $f$ be a continuous function on $G/P$. With the notation as above we have
    \[
		\lim_{T\to\infty}M_{\Gamma_T}^f=m(f)P_{\mathbbm 1_{G/P}}
	\]
in the weak operator topology of ${\mathcal B}\bigl(L^{2}(G/P,\nu)\bigr)$.	
That is, for any $\varphi,\psi\in L^{2}(G/P,\nu)$
	\[
		\lim_{T\to\infty}
		\frac{1}{|\Gamma_T|}\sum_{\gamma\in\Gamma_T}
		f(\bold {b}(\gamma))\frac{\langle\lambda_{G/P}(\gamma)\varphi,\psi\rangle}{\Xi(\gamma)}=
		\langle\varphi,\mathbbm 1_{G/P}\rangle\langle f,\psi\rangle.
	\]
\end{theo}

\subsection{Main results and ideas}
Among quasi-regular representations defined by  measurable actions of a locally compact second-countable group $G$ on a second-countable Borel space $B$ preserving the class of a  Borel probability measure $\nu$, we single out in Definition \ref{def: ergodic quasi-regular representation}  those representations which are \emph{ergodic with respect to some family $(E_n,e_n)$ where each $E_n\subset G$  is relatively compact Borel  and   $e_n:E_n\to B$ is Borel, and with respect to some function $f\in L^{\infty}(B,\nu)$.}

The definitions are coined so that a quasi-regular  representation  which is ergodic with respect to a family $(E_n,e_n)$ and sufficiently many  functions $f\in L^{\infty}(B,\nu)$  has to be irreducible; see Proposition \ref{prop: ergodic representations are irreducible}. The existence  of such a family implies the equidistribution of the sets $e_n(E_n)$ in $(B,\nu)$; see Remark \ref{rema: ergodic implies equidistribution}.   
The ergodicity of the quasi-regular representation associated to an action is strictly stronger than the ergodicity of the action; see the example after Remark \ref{rema: ergodic implies equidistribution}.

We show that the ergodicity of a quasi-regular representation  does not depend too much on the chosen family $E_n\subset G$:  one is free to work with balls, annuli, cones, sub-cones; see Subsection \ref{subsection: ergodicity relative to balls, annuli, cones and subcones}.

Theorem \ref{theorem: ergodicity of some quasi-regular representations}
gives sufficient conditions for the quasi-regular representation of a unimodular locally compact second-countable group $G$, endowed with a length function and acting on a Borel metric probability space $(B,d,\nu)$, to be ergodic
with respect to a symmetric family  $E_n\subset G$, Borel maps $e_n:E_n\to B$ and functions $f\in L^{\infty}(B,\nu)$ belonging to the closure of the subspace generated by characteristic functions of Borel subsets  $U\subset B$ such that $\nu(\partial U)=0$. 

As explained in the previous section, if $G$ is a connected non-compact semisimple Lie group $G$ with finite center, with maximal compact subgroup $K$, Cartan decomposition $K\overline{A^+}K$, minimal parabolic $P$ and $M=Z(A)\cap K$, we consider its Furstenberg-Poisson boundary $B=G/P=K/M$ and the boundary map
$$\bold b:KA^+K\to K/M$$ defined a.e. on $G$ as $\bold b(k_1ak_2)=k_1M$. If the real rank of $G$ equals one, the Furstenberg-Poisson boundary is identified with the visual boundary
of the symmetric space $G/K$, the boundary map $\bold b$ is defined on $G\setminus K$ and coincides  with the map $\bold b(\gamma)=c_{\gamma x_0}(\infty)$
explained above. If $\Gamma\subset G$ is a lattice, we apply Theorem \ref{theorem: ergodicity of some quasi-regular representations} to the intersections of $\Gamma$ with some  cones made bi-$K$-invariant in $G$ and deduce (using Theorem \ref{theo: ergodicity relative to cones of the quasi-regular representation of a lattice in a semisimple Lie group}) 
the ergodicity of the quasi-regular representation of a lattice in a semisimple Lie group (Theorem \ref{theo: ergodicity of the quasi-regular representation of a lattice in a semisimple Lie group}); the 
boundary maps $\bold b_T$ are the restrictions of   $\bold b(k_1ak_2)=k_1M$ to the $T$-truncation of (the image by the exponential map of) a Weyl chamber made bi-$K$-invariant.
The irreducibility of the restriction of the quasi-regular representation $\lambda_{G/P}$ to any lattice  follows from a general result in~\cite{CowSte}; according to Proposition~\ref{prop: ergodic representations are irreducible} this is also a straightforward corollary of Theorem~\ref{theo: ergodicity of the quasi-regular representation of a lattice in a semisimple Lie group}. 

Applying Theorem~\ref{theo: ergodicity of the quasi-regular representation of a lattice in a semisimple Lie group} or Bader-Muchnik's theorem to the case $$\mathrm{rank}_{\mathbb{R}} G=1$$ brings essentially the same information, although the original statement of Bader and Muchnik applies only  to uniform lattices.

Specializing  Theorem \ref{theo: ergodicity of the quasi-regular representation of a lattice in a semisimple Lie group} by choosing twice the vector $$\mathbbm 1_{G/P}\in L^2(G/P)$$ and applying the definition of the WOT one recovers the equidistribution of the radial $K$-component of the Cartan decomposition $KA^+K$ of lattice points in the Furstenberg-Poisson boundary: 
\emph{for every continuous function $f$ on $G/P$
\[
\lim_{n\to\infty}\frac{1}{|B_n|}\sum_{\gamma\in B_n}f(\bold{b}(\gamma))=\int_{G/P}f(b){\rm d}\nu(b).
\]} 
The analogous statement with the Iwasawa decomposition $G=KAN$ instead of the Cartan decomposition  and the  boundary map $\bold b(kan)=kM$ instead of  $\bold b(k_1ak_2)=k_1M$ follows from \cite[Theorem 1]{GorMau}. We refer the reader to \cite[Theorem 7.2]{GorNev} for  the equidistribution  of both $K$-components of the Cartan decomposition $KA^+K$  of lattice points with an explicit control on the rate of convergence  for Lipschitz functions. 

Specializing Theorem \ref{theo: ergodicity of the quasi-regular representation of a lattice in a semisimple Lie group} by choosing $f=\mathbbm 1_{G/P}\in L^{\infty}(G/P)$ and applying the results of Subsection \ref{subsection: ergodicity relative to balls, annuli, cones and subcones} confirms Breuillard's guess mentioned above.

\subsection{Key points from the proofs, questions and speculations} The proof of Theorem \ref{theorem: ergodicity of some quasi-regular representations}
splits into two parts. The first part consists in ``the identification of the limit'' -- by which we  mean the convergence
\[
	\lim_{n\to\infty}\langle M_{(E_n,e_n)}^{\mathbbm 1_U}\mathbbm 1_V,\mathbbm 1_W\rangle=\nu(U\cap W)\nu(V)
\]
(see Formula \ref{formula: identification of the limit}  below and Section \ref{section: ergodic properties of quasi-regular representations} below for notation)  -- which is established for sufficiently many characteristic functions.  The second part amounts to prove the uniform bound 
\[
	\sup_n \|M_{E_n}^{\mathbbm 1_B}\|_{op}<\infty
\]
for the family of operators $M_{E_n}^{\mathbbm 1_B}\in {\mathcal B}\bigl(L^2(B,\nu)\bigr)$. Assuming both parts proved, the main conclusion of Theorem~\ref{theorem: ergodicity of some quasi-regular representations}, that is the ergodicity of the representation $\pi_{\nu}$ is easily deduced from density arguments and the following lemma whose proof is straightforward:
\begin{lemma}\label{lemma: straightforward}(From $L^{\infty}$ to $L^2$.) Let ${\mathcal H}$ be a Hilbert space and let $V$ be a dense subspace. Let $A_n$ be a uniformly bounded sequence of ${\mathcal B}({\mathcal H})$.
If for all $v,w\in V$ 
	\[
		\lim_{n\to\infty}\langle A_nv,w\rangle=0,
	\]
then 
	\[
		\lim_{n\to\infty}A_n=0
	\]
in the weak operator topology.	
\end{lemma}
This twofold approach is a core idea in \cite{BadMuc}; it  works in strictly negative curvature. We show  that it also works in higher rank. Theorem~\ref{theorem: ergodicity of some quasi-regular representations} may be viewed as an attempt to formalize this  twofold approach in a general frame, flexible enough to cover different geometric situations.  Although semisimple Lie groups over non-Archimedean local fields and their lattices and $S$-arithmetic groups are  not touched in this paper, we believe that a similar approach applies to those groups as well. Also it would be interesting to obtain effective forms of convergence with a good control on the error term. (Reference \cite{GorNev} looks helpful for both purposes.)
We now briefly indicate  how the ``identification of the limit'', respectively the ``uniform boundedness'', is obtained in illustrative cases.
\subsubsection{Identification of the limit} As explained in Subsection~\ref{subsection: Sjogren} below, in the case of symmetric spaces the first hypothesis of Theorem  \ref{theorem: ergodicity of some quasi-regular representations}, namely the inequality
	\[
		\frac{\langle\pi_{\nu}(g)\mathbbm 1_B,\mathbbm 1_{\{b\in B : d(b,e_n(g))\geq r\}}‚\rangle}{\Xi(g)}\leq h_r(L(g)),
	\]
follows from Fatou type theorems (see \cite{Sjö}, \cite{SjöAnnals} and \cite{Sch}) for the normalized square root of the Poisson kernel 	
$$
(\mathcal{P}_0\varphi)(g)=\frac{\langle\pi_{\nu}(g)\mathbbm 1_B,\overline{\varphi}\rangle}{\Xi(g)}.
$$
An important condition in generalized Fatou theorems is the convergence type (weak and/or restricted tangential, admissible). We handle this technical issue by first proving the ergodicity of $\pi_{\nu}$ with respect to sub-cones of a Weyl chamber in Theorem \ref{theo: ergodicity relative to cones of the quasi-regular representation of a lattice in a semisimple Lie group}. The restrictions imposed by the sub-cones may then be removed thanks to the results of Subsection~\ref{subsection: ergodicity relative to balls, annuli, cones and subcones}.
As shown in \cite[Lemma 5.1, Prop. 7.4]{BoyCAT}, the above inequality is true for discrete isometry groups of  proper $\mathrm{CAT}(-1)$ spaces, with $\delta$-Ahlfors regular Patterson-Sul\-li\-van conformal density of dimension $\delta$,  and although the proof relies on geometric properties of negatively curved spaces it could be deduced  from a Fatou type theorem for the normalized square-root of the Poisson kernel \cite[Prop. 7.4]{BoyCAT}.   It would be interesting to develop  the boundary theory with Fatou type theorems for the normalized square-root of the Poisson kernels for more $\mathrm{CAT}(0)$ spaces and groups. 

The second hypothesis of Theorem~\ref{theorem: ergodicity of some quasi-regular representations}, namely the inequality
	\[
		\limsup_{n\to\infty}\frac{1}{\mathrm{vol}(E_n)}\int_{E_n}\mathbbm 1_U(e_n(g^{-1}))\mathbbm 1_V(e_n(g)){\rm d}g\leq\nu(U)\nu(V),
	\] 
is satisfied thanks to  counting results. For closed negatively curved manifolds it follows from a result which goes back to G. Margulis' thesis \cite[Theorem C.1]{BadMuc}. For $\mathrm{CAT}(-1)$ spaces and  groups with finite Bowen-Margulis-Sullivan measure and non-arithmetic spectrum it follows (see \cite{BoyCAT}) from a result of T. Roblin \cite[Thm. 4.1.1]{Rob}. For free groups, a direct counting gives the result \cite{BoyPin} and it would be interesting to understand how the normalizing constant in \cite[Thm. 4.1.1]{Rob} varies with $\alpha>0$ in the case of the tree covering the wedge of two circles of lengths $1$ and $\alpha$.  For locally symmetric spaces of finite volume it follows -- as explained below in Subsection~\ref{subsection: counting lattice points and the wave-front lemma} 
-- from the wave-front lemma and counting results from \cite{GorOh}. 

\subsubsection{Uniform boundedness} Bounding operator norms of averages of unitary operators may be a difficult problem; this is well illustrated
by Valette's conjecture about property RD for uniform lattices in higher rang semisimple Lie groups (see \cite{Cha}). In \cite{BadMuc}, Bader and Muchnik apply the Riesz-Thorin interpolation theorem and reduce their $L^2$ uniform bound problem to one about $L^{\infty}$ norms. Following their idea we explain in Subsection \ref{subsection: uniform bounds for averages of unitary operators} below how the problem boils down to bounding  cocyles averages of the type
\[
	M_E^{\mathbbm 1_B}\mathbbm 1_B(b)=\frac{1}{\mathrm{vol}(E)}\int_E\frac{c(g,b)^{1/2}}{\Xi(g)}{\rm d}g.
\]
In $\mathrm{CAT}(-1)$ spaces, the concept of sampling from \cite{BadMuc} and geometric inequalities \cite[4.4 Uniform boundedness]{BoyCAT} lead to the wanted uniform upper bound for the $L^{\infty}$ norms. (An alternative proof of the uniform bound on the operator norms consists in applying property RD to the quasi-regular representation. This  works for Gromov hyperbolic groups; see \cite[Uniform boundedness 4.4]{BoyMay}.) For lattices in semisimple Lie groups the upper bound is proved in two steps. The first step (Proposition \ref{proposition: K-invariance}) is an explicit computation of the cocycle averages: they are all equal to $1$ (this was first observed in \cite[Lemma 2.6]{BoyPin}). The second step (Proposition \ref{proposition: uniform boundedness}), which goes back to \cite{BoyThesis}, is a discretization of the cocycles averages; it works because the cocycles and the Harish-Chandra functions are stable in the sense of Subsection \ref{subsection: stable functions on locally compact groups}.

\subsection{Acknowledgements}
We are grateful to  H. Oh for explanations about her paper with A. Gorodnik \cite{GorOh}. Sharp criticism by A. Nevo enabled us to correct  inaccuracies in the exposition of a previous version; we are very grateful to him. We thank  J.-F. Lafont who mentioned to us reference \cite{Ricks}. We are very grateful to U. Bader for conversations, talks, notes, preprints, and to E. Breuillard for sharing his notes on property RD with us, for suggestions improving the quality of the exposition, and for mentioning to us reference \cite{GorMau} a long time ago.

\section{Ergodic properties of quasi-regular representations}\label{section: ergodic properties of quasi-regular representations}

\subsection{Quasi-regular representations}\label{subsection: The quasi-regular representation and the Harish-Chandra function associated to a quasiinvariant measure}
Let $B$ be a second-countable topological space and let $\nu$ be a Borel probability measure on $B$. 
Let $G$ be a locally compact second-countable group. Let ${\rm d}g$ denote a left Haar measure on $G$. 
We assume that $G$ acts continuously on $B$ and that  the action preserves sets of $\nu$-measure zero.
For all $\varphi\in L^2(B,\nu)$ and all $g\in G$ we have
\[
	\int_G\varphi(gb)c(g,b){\rm d}\nu(b)=\int_G\varphi(b){\rm d}\nu(b),
\]
where 
\[
	c(g,b)=\frac{{\rm d}g^{-1}_*\nu(b)}{{\rm d}\nu(b)}
\]
is the  Radon-Nikodym derivative at the point $b$ of the transformation defined by $g^{-1}$. The formula
\[
	(\pi_{\nu}(g)\varphi)(b)=\varphi(g^{-1}b)c(g^{-1},b)^{1/2}
\]
defines a unitary representation 
\[
	\pi_{\nu}:G\to \mathcal{U}(L^2(B,\nu))
\]
called the \emph{quasi-regular representation} associated to the action of $G$ on $(B,\nu)$.
(See \cite[Proposition A.6.1 and Lemma A.6.2]{BHV} for the strong continuity of $\pi_{\nu}$.)
We will always assume that for any relatively compact open set $U\subset G$
\[
	\sup_{g\in U}\|\frac{{\rm d}g_*\nu}{{\rm d}\nu}\|_{L^{\infty}(B,\nu)}<\infty,
\]
equivalently, for any relatively compact open set $U\subset G$
\[
	\sup_{g\in U}\|\pi_{\nu}(g)\mathbbm 1_B\|_{L^{\infty}(B,\nu)}<\infty.
\]
This condition says that there is a uniform bound on the contraction (or dilatation) of any Borel set in $B$ under the action of any compact part of $G$. It is satisfied in each  example mentioned in the introduction: the boundary $B$ is a compact space and  the  Radon-Nykodym derivative is given by a cocycle $c(g,b)$ continuous in $(g,b)$. 

The \emph{Harish-Chandra function} $\Xi$ associated to the action of $G$ on $(B,\nu)$ is the coefficient of $\pi_{\nu}$ defined by the characteristic function $\mathbbm 1_B\in L^2(B,\nu)$
of $B$. Namely
\[
	\Xi(g)=\langle\pi_{\nu}(g)\mathbbm 1_B,\mathbbm 1_B\rangle=\int_Bc(g^{-1},b)^{1/2}{\rm d}\nu(b).
\]
Let $\varphi\in L^1(B,\nu)$.
In some special cases (see Subsection \ref{subsection: Sjogren} below) it is fruitful to consider
$$
(\mathcal{P}_0\varphi)(g)=\frac{\langle\pi_{\nu}(g)\mathbbm 1_B,\overline{\varphi}\rangle}{\Xi(g)}
$$
as \emph{the normalized square root of a Poisson kernel}. It brings a useful link with potential theory. 
 
\subsection{Averages of unitary operators}
Let $E\subset G$ be a relatively compact  Borel subset and let 
\[
\mathrm{vol}(E)=\int_G\mathbbm 1_E(g){\rm d}g.	
\]
Let $e:E\to B$ be a Borel map.
If $\pi_{\nu}$ is as above, let ${\mathcal B}\bigl(L^2(B,\nu)\bigr)$ denote the Banach space of the bounded operators on $L^2(B,\nu)$.
Let $f:B\to\mathbb C$ be a Borel bounded function. If $\mathrm{vol}(E)\neq 0$ we may consider the average 
	\[
		M_{(E,e)}^f=\frac{1}{\mathrm{vol}(E)}\int_{E}f(e(g))\frac{\pi_{\nu}(g)}{\Xi(g)}{\rm d}g.
	\]
As the function 
\[
	g\mapsto\frac{\mathbbm 1_{E}(g)f(e(g))}{\mathrm{vol}(E)\Xi(g)}
\]
belongs to $L^1(G,{\rm d}g)$ it follows (see \cite[Ch. XI. § 7. 25.1-2]{God}) that
$M_{(E,e)}^f\in {\mathcal B}\bigl(L^2(B,\nu)\bigr)$. Let $P_{\mathbbm 1_B}\in {\mathcal B}(L^2(B,\nu))$ be the orthogonal pro\-jec\-tion
onto the complex line of constant functions on $B$ and let $m(f)\in {\mathcal B}\bigl(L^2(B,\nu)\bigr)$	be the multiplication operator defined by $f$.

\begin{defi}(Ergodic quasi-regular representations.)\label{def: ergodic quasi-regular representation} For each $n\in\mathbb N$ let $(E_n,e_n)$ be as above. We say that the representation $\pi_{\nu}$ is ergodic relative to the family $(E_n,e_n)$
	and the map $f$ if 
	\[
		\lim_{n\to\infty}M_{(E_n,e_n)}^f=m(f)P_{\mathbbm 1_B}
	\]
in the weak operator topology. That is, if for all $\varphi,\psi\in L^2(B,\nu)$
	\[
		\lim_{n\to\infty}\frac{1}{\mathrm{vol}(E_n)}\int_{E_n}f(e_n(g))\frac{\langle\pi_{\nu}(g)\varphi,\psi\rangle}{\Xi(g)}{\rm d}g=\langle\varphi,\mathbbm 1_B\rangle\langle f,\psi\rangle.
	\]
\end{defi}

\begin{theo}\label{theorem: ergodicity of some quasi-regular representations}(Fatou, counting, and uniform boundedness imply ergodicity of the representation.) Let $\pi_{\nu}:G\to \mathcal{U}\bigl(L^2(B,\nu)\bigr)$ be a quasi-regular representation such that for any relatively compact open set $U\in G$
$$\sup_{g\in U}\|\pi_{\nu}(g)\mathbbm 1_B\|_{L^{\infty}(B,\nu)}<\infty.$$
Suppose $G$ is unimodular. Let $L$ be a length function on $G$ and let $d$ be a distance on $B$ inducing the topology of $B$. Let $E_n=E_n^{-1}$ be a sequence of relatively compact Borel subsets  of $G$ such that
$\lim_{n\to\infty}\mathrm{vol}(E_n)=\infty$. Let $e_n:E_n\to B$ be a sequence of Borel maps.
Assume the following two conditions hold:
\begin{enumerate}
	\item for each $r>0$ there is a non-increasing function $h_r:[0,\infty[\to[0,\infty[$ such that $\lim_{s\to\infty}h_r(s)=0$ and such that for all $n\in\mathbb N$ and for all $g\in E_n$
	\[
		\frac{\langle\pi_{\nu}(g)\mathbbm 1_B,\mathbbm 1_{\{b\in B : d(b,e_n(g))\geq r\}}‚\rangle}{\Xi(g)}\leq h_r(L(g)),
	\]
	\item for all Borel subsets $U,V\subset B$ such that $\nu(\partial U)=\nu(\partial V)=0$
	\[
		\limsup_{n\to\infty}\frac{1}{\mathrm{vol}(E_n)}\int_{E_n}\mathbbm 1_U(e_n(g^{-1}))\mathbbm 1_V(e_n(g)){\rm d}g\leq\nu(U)\nu(V).
	\] 
\end{enumerate}
\emph{Then for all Borel subsets $U,V,W\subset B$ such that $\nu(\partial U)=\nu(\partial V)=\nu(\partial W)=0$
\begin{equation}\label{formula: identification of the limit}
	\lim_{n\to\infty}\langle M_{(E_n,e_n)}^{\mathbbm 1_U}\mathbbm 1_V,\mathbbm 1_W\rangle=\nu(U\cap W)\nu(V).
\end{equation}}

Assume moreover 
\[
	\sup_n \|M_{E_n}^{\mathbbm 1_B}\mathbbm 1_B\|_{\infty}<\infty.
\]
Let $H\subset L^{\infty}(B,\nu)$ be the smallest subspace containing all characteristic functions $\mathbbm 1_U$ with $U\subset B$ Borel such that $\nu(\partial U)=0$ and let $\overline{H}\subset L^{\infty}(B,\nu)$ denote its closure. 

\emph{Then $\pi_{\nu}$ is ergodic with respect the family $(E_n,e_n)$ and any $f\in\overline{H}$.}
\end{theo}

\begin{rema}
Notice that if $f=\mathbbm 1_B$ then the choices of the maps $e_n$ are irrelevant; in this case we say that the representation $\pi_{\nu}$ is ergodic relative to the sets $E_n$. We have
	\[
		\lim_{n\to\infty}M_{E_n}=P_{\mathbbm 1_B},
	\]
hence $P_{\mathbbm 1_B}$ belongs to the von Neumann algebra generated by $\pi_{\nu}(G)$ in the algebra ${\mathcal B}\bigl(L^2(B,\nu)\bigr)$.
\end{rema}

\begin{rema}
The definition of an ergodic quasi-regular representation generalizes the definition of an ergodic  measure preserving transformation. Indeed, if $T$ is a measure preserving transformation of the probability space $(B,\nu)$, the associated quasi-regular representation $\pi_{\nu}:\mathbb Z\to \mathcal{U}(L^2(B,\nu))$ is the Koopman representation,
the constant function $\mathbbm 1_B$ is fixed by $\pi_{\nu}$, hence $\Xi$ is constant equal to $1$, the choice $E_n=\{k\in\mathbb Z : 0\leq k\leq n-1\}$ defines the  Birkhoff sum $M_{E_n}$, and it is well-known that 
		\[
			\lim_{n\to\infty}M_{E_n}=P_{\mathbbm 1_B}
		\]
in the weak operator topology if and only if the measure preserving transformation $T$ is ergodic.
\end{rema}

\begin{prop}\label{prop: ergodic representation  implies ergodic action} If the quasi-regular representation $\pi_{\nu}$ of $G$ on $L^2(B,\nu)$ is ergodic relative to a family $E_n$, then the action of $G$ on $B$ is ergodic. 
\end{prop}
\begin{proof} Let $E$ be a Borel subset of $B$ such that for all $g\in G$ we have $gE=E$.
	Let $E^c$ denote the complement of $E$ in $B$. Obviously
	\[
		\langle\pi_{\nu}(g)\mathbbm 1_E,\mathbbm 1_{E^c}\rangle=0
	\]
	for all $g\in G$.
	Hence, if $\pi_{\nu}$ is ergodic relative to a family $E_n$, then
	\[
		0=\langle M_{E_n}\mathbbm 1_E,\mathbbm 1_{E^c}\rangle\to \langle P_{\mathbbm 1_B}\mathbbm 1_E,\mathbbm 1_{E^c}\rangle=\nu(E)\nu(E^c).
	\]
\end{proof}

\begin{rema}
If $\pi_{\nu}$ is ergodic relative to a family of couples $(E_n,e_n)$ and a Borel bounded function $f$, then
	\[
		\lim_{n\to\infty}M_{(E_n,e_n)}^f\mathbbm 1_B=f
	\]
in the weak topology of $L^2(B,\nu)$. Hence if this holds for 	
a family of Borel bounded functions $f$ which is dense in $L^2(B,\nu)$, then $\mathbbm 1_B$ is a cyclic vector of $\pi_{\nu}$.
\end{rema}

\begin{rema}\label{rema: ergodic implies equidistribution}
If $\pi_{\nu}$ is ergodic relative to a family of couples $(E_n,e_n)$ for all continuous $f$, we may specialize the convergence to the case $\varphi=\psi=\mathbbm 1_B$ to obtain the equidistribution of $e_n(E_n)\subset B$ relative to $\nu$:
\[
	\lim_{n\to\infty}\frac{1}{\mathrm{vol}(E_n)}\int_{E_n}f(e_n(g)){\rm d}g=\int_Bf(b){\rm d}\nu(b),
\]
or, in the case $G=\Gamma$ is discrete and the Haar measure is the counting measure
\[
	\lim_{n\to\infty}\frac{1}{|E_n|}\sum_{\gamma\in E_n}f(e_n(\gamma))=\int_Bf(b){\rm d}\nu(b).
\]
\end{rema}

If $E\subset B$ is a Borel $G$-invariant subset with $0<\nu(E)<1$, then the orthogonal projection $P_{\mathbbm 1_E}$  commutes with $\pi_{\nu}(G)$, 
so $\mathrm{ker}(P_{\mathbbm 1_E})$ is a non trivial closed invariant subspace for $\pi_{\nu}$. Hence \emph{the ergodicity of the action is a necessary condition for the irreducibility of $\pi_{\nu}$. It is not a sufficient one}: the diagonal action of $SL(2,\mathbb Z)$ on the product of two copies of the circle at infinity of $SL(2,\mathbb R)/SO(2,\mathbb R)$ equipped with the product of the angular measure is ergodic (see \cite[Theorem 17]{Kai} and the introduction of \cite{ConMuc}) but the involution permuting the coordinates of the circles at infinity defines a non trivial invariant subspace of the associated quasi-regular representation. The next proposition shows that \emph{the ergodicity of the representation is strictly stronger than the ergodicity of the action because it implies the irreducibility of the representation}.
 
\begin{prop} (Ergodic representations are irreducible.)\label{prop: ergodic representations are irreducible} 
Assume that for each bounded Borel function $f: B\to\mathbb C$ from a dense family in $L^2(B,\nu)$ containing $\mathbbm 1_B$ there is a sequence
$(E_n,e_n)$ relative to which $\pi_{\nu}$ is ergodic. Then $\pi_{\nu}$ is irreducible.
\end{prop}

\begin{proof} As mentioned in the above remarks, the vector $\mathbbm 1_B$ is cyclic and the orthogonal projection $P_{\mathbbm 1_B}$ onto the complex line it generates belongs to the von Neumann algebra generated by $\pi_{\nu}(G)$. The irreducibility is then a consequence of the following classical lemma which is easy to prove.
\end{proof}

\begin{lemma}\label{lemma: irreducible} 
Let $\pi$ be a unitary representation on a Hilbert space ${\mathcal H}$.
Suppose there exists a cyclic vector $w\in {\mathcal H}$ such that the orthogonal projection $P_w$ onto the line it generates belongs to the von Neumann algebra generated by $\pi(G)$.
Then $\pi$ is irreducible.	
\end{lemma}

\subsection{Ergodicity relative to balls, annuli, cones and subcones}\label{subsection: ergodicity relative to balls, annuli, cones and subcones}

\begin{prop}\label{prop: ergodicity relative to cones and subcones}(Ergodicity relative to cones and subcones.) Let $\pi_{\nu}$ be a quasi-regular representation of $G$ as above.
Let $f:B\to\mathbb C$ be a Borel bounded map.	
Assume that for each $n\in\mathbb N$  we are given relatively compact Borel subsets $B_n$ and $C_n\subset B_n$ of $G$ of non zero Haar measure and Borel maps $b_n:B_n\to B$ and $c_n=b_n|_{C_n}:C_n\to B$. Assume that 
$$\lim_{n\to\infty}\frac{\mathrm{vol}(B_n\setminus C_n)}{\mathrm{vol}(B_n)}=0.$$
Assume moreover that the operators $M^f_{(B_n,b_n)}$ and  $M^f_{(C_n,c_n)}$ are uniformly bounded.
Then the following conditions are equivalent:
\begin{enumerate}
\item the representation $\pi_{\nu}$ is ergodic relative to the couples 
 $(B_n,b_n)$ and the map $f$,
\item the representation $\pi_{\nu}$ is ergodic relative to the couples 
 $(C_n,c_n)$ and the map $f$.
\end{enumerate}	
\end{prop}

\begin{proof} As $L^{\infty}(B,\nu)\subset L^2(B,\nu)$ is dense and the operators are uniformly bounded Lemma \ref{lemma: straightforward} applies. Hence it is enough to prove that for all $\varphi, \psi\in L^{\infty}(B,\nu)$
\[
\lim_{n\to\infty}\langle(M_{(B_n,b_n)}^f-M_{(C_n,c_n)}^f)\varphi,\psi\rangle=0.
\] 
We have
\begin{align*}
	M_{(B_n,b_n)}^f-M_{(C_n,c_n)}^f&=\frac{1}{\mathrm{vol}(B_n)}\int_{B_n\setminus C_n}f(b_n(g))\frac{\pi_{\nu}(g)}{\Xi(g)}{\rm d}g\\
	&+\frac{\mathrm{vol}(C_n)-\mathrm{vol}(B_n)}{\mathrm{vol}(B_n)\mathrm{vol}(C_n)}\int_{C_n}f(c_n(g))\frac{\pi_{\nu}(g)}{\Xi(g)}{\rm d}g.
\end{align*}
Writing each bounded function $f,\varphi,\psi$ as the sum of its positive and negative real part and  its  positive and negative imaginary part we obtain
\begin{align*}
	&|\langle(M_{(B_n,b_n)}^f-M_{(C_n,c_n)}^f)\varphi,\psi\rangle|\leq\\ 
	&\frac{4^3\|f\|_{\infty}\|\varphi\|_{\infty}\|\psi\|_{\infty}}{\mathrm{vol}(B_n)}\int_{B_n\setminus C_n}\frac{\langle\pi_{\nu}(g)\mathbbm 1_B,\mathbbm 1_B\rangle}{\Xi(g)}{\rm d}g\\
	&+4^3\|f\|_{\infty}\|\varphi\|_{\infty}\|\psi\|_{\infty}
	\frac{\mathrm{vol}(B_n)-\mathrm{vol}(C_n)}{\mathrm{vol}(B_n)\mathrm{vol}(C_n)}\int_{C_n}\frac{\langle\pi_{\nu}(g)\mathbbm 1_B,\mathbbm 1_B\rangle}{\Xi(g)}{\rm d}g\\
	&=4^3\|f\|_{\infty}\|\varphi\|_{\infty}\|\psi\|_{\infty}2\frac{\mathrm{vol}(B_n)-\mathrm{vol}(C_n)}{\mathrm{vol}(B_n)}.
\end{align*} 	
\end{proof}

\begin{prop}(Ergodicity relative to balls and annuli.)\label{prop: ergodicity relative to balls and annuli} Let $\pi_{\nu}$ be a quasi-regular representation of $G$ as above.
Let $f:B\to\mathbb C$ be a Borel bounded map.	
Assume that for each $n\in\mathbb N$  we are given relatively compact Borel subsets $B_n$ and $C_n$ of $G$ of non zero Haar measure and Borel maps $b_n:B_n\to B$ and $c_n:C_n\to B$. 
Assume that for each $n$ the set 
$$B_n=\bigsqcup_{k=1}^nC_k$$
is the disjoint union of the sets $C_1,\dots,C_k,\dots,C_n$ and that $b_n$ restricted to $C_k$ equals $c_k$.
\begin{enumerate}
		\item Assume $\lim_{n\to\infty}\frac{\mathrm{vol}(B_{n-1})}{\mathrm{vol}(B_n)}$ exists and is strictly smaller than $1$. If $\pi_{\nu}$ is ergodic relative to the $(B_n,b_n)$ and $f$, then it is ergodic relative to the $(C_n,c_n)$ and $f$.
		\item Assume $\lim_{n\to\infty}\mathrm{vol}(B_n)=\infty$. If $\pi_{\nu}$ is ergodic relative to the $(C_n,c_n)$ and $f$, then it is ergodic relative to the $(B_n,b_n)$ and $f$.
\end{enumerate}	
\end{prop}
\begin{proof} Let $\varphi,\psi\in L^2(B,\nu)$.
	
Suppose $\lim_{n\to\infty}\frac{\mathrm{vol}(B_{n-1})}{\mathrm{vol}(B_n)}=c<1$ and that $\pi_{\nu}$ is ergodic relative to the $(B_n,b_n)$ and $f$.
We have
\[
\langle M_{(B_n,b_n)}^f\varphi,\psi\rangle=\frac{\mathrm{vol}(B_{n-1})}{\mathrm{vol}(B_n)}\langle M_{(B_{n-1},b_{n-1}\rangle}^f\varphi,\psi\rangle+\frac{\mathrm{vol}(C_n)}{\mathrm{vol}(B_n)}\langle M_{(C_n,c_n)}^f\varphi,\psi\rangle,	
\]
hence letting $n$ tend to infinity we obtain
\[
	\langle m(f)P_{\mathbbm 1_B}\varphi,\psi\rangle=c(m(f)P_{\mathbbm 1_B}\varphi,\psi)+(1-c)\lim_{n\to\infty}\langle M_{(C_n,c_n)}^f\varphi,\psi\rangle.
\]
This proves the ergodicity of  $\pi_{\nu}$  relative to the $(C_n,c_n)$ and $f$ because by hypothesis $c\neq 1$.

Suppose $\lim_{n\to\infty}\mathrm{vol}(B_n)=\infty$ and that  $\pi_{\nu}$ is ergodic relative to the $(C_n,c_n)$ and $f$.
We have
\[
	\langle M_{(B_n,b_n)}^f\varphi,\psi\rangle=\frac{1}{\mathrm{vol}(B_n)}\sum_{k=1}^n\mathrm{vol}(C_k)\langle M_{(C_k,c_k)}^f\varphi,\psi\rangle.
\]
Let $a_k=\langle M_{(C_k,c_k)}^f\varphi,\psi\rangle$ and let $a=\langle m(f)P_{\mathbbm 1_B}\varphi,\psi\rangle$. By hypothesis 
$\lim_{k\to\infty}a_k=a$. Let $s_k=\mathrm{vol}(C_k)$. The proof is finished thanks to the following lemma whose proof is straightforward.
\end{proof}
\begin{lemma} Let $a_k$ be a sequence of complex numbers which converges to $a$.
	Let $s_k$ be a sequence of non-negative numbers such that $\sum_{k=1}^{\infty}s_k=\infty$.
	Then
	\[
		\lim_{n\to\infty}\frac{1}{\sum_{k=1}^ns_k}\sum_{k=1}^ns_ka_k=a.
	\] 
\end{lemma}
\subsection{Uniform bounds for averages of unitary operators}\label{subsection: uniform bounds for averages of unitary operators}
The following proposition is a corollary of the Riesz-Thorin theorem.
\begin{prop}\label{proposition: Riesz Thorin}
Let $(B,\nu)$ be a probability space and let
\[
	T_1:L^1(B,\nu)\to L^1(B,\nu)
\]
be a bounded operator such that its restriction $T_2$ to $L^2(B,\nu)$ preserves $L^2(B,\nu)$
and its restriction $T_{\infty}$ to $L^{\infty}(B,\nu)$ preserves $L^{\infty}(B,\nu)$. If $T_2$ is bounded and self-adjoint then 
$T_{\infty}$ is bounded,
\[
	\|T_{\infty}\|_{L^{\infty}(B,\nu)\to L^{\infty}(B,\nu)}=\|T_1\|_{L^1(B,\nu)\to L^1(B,\nu)},
\]
\[
	\|T_{2}\|_{L^2(B,\nu)\to L^2(B,\nu)}\leq\|T_{\infty}\|_{L^{\infty}(B,\nu)\to L^{\infty}(B,\nu)}.
\] 
\end{prop}

Let $E$ be a Borel subset of $G$ of positive measure. Let $b\in B$. Recall that by definition
\[
	M_E^{\mathbbm 1_B}\mathbbm 1_B(b)=\frac{1}{\mathrm{vol}(E)}\int_E\frac{c(g,b)^{1/2}}{\Xi(g)}{\rm d}g.
\]
Let $e:E\to B$ be a Borel map and $f:B\to\mathbb C$  a bounded Borel map.
\begin{lemma}\label{lemma: averaging the cocycle at the point b} Assume $E$ is relatively compact and symmetric. If $G$ is unimodular then
	\[
		\|M_{(E,e)}^f\|_{L^2\to L^2}\leq\|f\|_{\infty}\|M_E^{\mathbbm 1_B}\mathbbm 1_B\|_{\infty}.lar
	\]
\end{lemma}
\begin{proof} To save notation we write $M_E$ instead of $M_E^{\mathbbm 1_B}$. As $E$ is relatively compact in $G$ locally compact second countable and as we always assume that for any relatively compact open set $U\in G$
\[
	\sup_{g\in U}\|\frac{{\rm d}g_*\nu}{{\rm d}\nu}\|_{L^{\infty}(B,\nu)}<\infty,
\]	
it is easy to check that for $p\in\{1;2;\infty\}$
\[
	M_E=\frac{1}{\mathrm{vol}(E)}\int_E\frac{\pi_{\nu}(g)}{\Xi(g)}{\rm d}g\in {\mathcal B}\bigl(L^p(B,\nu)\bigr).
\]
Obviously, 
$\|M_{(E,e)}^f\|_{L^p\to L^p}\leq \|f\|_{\infty}\|M_E\|_{L^p\to L^p}$.
As $E=E^{-1}$ and as $G$ is unimodular we have $M_E=M_E^*$. According to  Proposition \ref{proposition: Riesz Thorin}
\[
	\|M_E\|_{L^2\to L^2}\leq\|M_E\|_{L^{\infty}\to L^{\infty}}=\|M_E\mathbbm 1_B\|_{\infty}.
\]
\end{proof}

\subsection{Stable functions on locally compact groups}\label{subsection: stable functions on locally compact groups}

\begin{defi} A function $f:G\to\mathbb C$ on a locally compact group is \emph{left-stable} if there exist a constant $C\geq 1$ and a neighborhood 
$V$ of the identity $e$ in $G$	such that for all $g\in G$ and  for all $v\in V$
\[
	|f(g)|/C\leq|f(vg)|\leq C|f(g)|.
\]
We write $f(g)\asymp f(vg)$  when we don't want to emphasize the actual value of the constant $C$.
\end{defi}
\begin{prop}\label{proposition: stable} Assume that the cocycle $c:G\times B\to[0,\infty[$ associated to an action of a locally compact group $G$ on a compact probability
	space $(B,\nu)$ is continuous. Let $\pi_{\nu}$ be the corresponding unitary representation. Then for any $b\in B$ the function
	$g\mapsto c(g,b)$ is left-stable. The generalized Harish-Chandra function $\Xi(g)=\langle\pi_{\nu}(g)\mathbbm 1_B,\mathbbm 1_B\rangle$ is left and right-stable.
\end{prop}
\begin{proof} 
	Let $V$ be a symmetric compact neighborhood of $e$ in $G$. Since $c$ is continuous, $c|_{V\times B}$  reaches its maximum $M$ and its minimum $m$ on the compact set $V\times B$.
	Let $v\in V$. The cocycle identity 
		\[
			1=c(v^{-1}v,b)=c(v^{-1},vb)c(v,b)
		\]
and the symmetry of $V$ imply that $Mm=1$.
Since $c(vg,b)=c(v,gb)c(g,b)$ we obtain for all $g\in G$ and all $v\in V$
\[
	\frac{c(g,b)}{\|c|_{V\times B}\|_{\infty}}\leq c(vg,b)\leq\|c|_{V\times B}\|_{\infty}c(g,b).
\]
We have
\begin{align*}
	\Xi(vg)&=\Xi((vg)^{-1})=\int_Bc(vg,b)^{1/2}{\rm d}\nu(b)\asymp \int_Bc(g,b)^{1/2}{\rm d}\nu(b)\\
	       &=\Xi(g^{-1})=\Xi(g),
\end{align*}
and hence
\[
	\Xi(gv)=\Xi(v^{-1}g^{-1})\asymp\Xi(g^{-1})=\Xi(g).
\]		
\end{proof}

\begin{lemma}\label{lemma: from discrete to continuous} Assume the cocycle $c:G\times B\to[0,\infty[$ associated to an action of a locally compact group $G$ on a compact probability
	space $(B,\nu)$ is continuous. Let $\pi_{\nu}$ be the corresponding unitary representation. Let $\Gamma$ be a discrete subgroup of $G$. There exist  a relatively compact neighborhood $U$ of $e$ in $G$ and a constant $C>0$ such that for any 
non-empty finite subset $\Lambda$ of $\Gamma$ and for any $b\in B$
\[
	\sum_{\gamma\in\Lambda}\frac{\pi_{\nu}(\gamma)\mathbbm 1_B(b)}{\Xi(\gamma)}\leq C \int_{\Lambda U}\frac{\pi_{\nu}(g)\mathbbm 1_B(b)}{\Xi(g)}{\rm d}g.
\]
\end{lemma}
\begin{proof} According to Proposition \ref{proposition: stable} the cocycle is left-stable.
Hence, as $g\mapsto g^{-1}$ is a homeomorphism fixing $e\in G$, there exists a constant $C'>0$ and a relatively compact neighborhood $U$ of $e$ in $G$ such that for all $u\in U$, all $g\in G$ and  all $b\in B$
	\[
		\pi_{\nu}(g)\mathbbm 1_B(b)=c(g^{-1},b)^{1/2}\leq C'c(u^{-1}g^{-1},b)^{1/2}=C'\pi_{\nu}(gu)\mathbbm 1_B(b).
	\]
According to the same proposition, the Harish-Chandra function is right-stable. Hence we may assume that for all $u\in U$ and all $g\in G$
\[
	\Xi(g)>\Xi(gu)/C'.
\]
It follows	that for any $\gamma\in\Gamma$
\[
\frac{\pi_{\nu}(\gamma)\mathbbm 1_B(b)}{\Xi(\gamma)}=\frac{1}{\mathrm{vol}(U)}\int_U\frac{\pi_{\nu}(\gamma)\mathbbm 1_B(b)}{\Xi(g)}\mathrm{d}u
\leq\frac{(C')^2}{\mathrm{vol}(U)} \int_U\frac{\pi_{\nu}(\gamma u)\mathbbm 1_B(b)}{\Xi(\gamma u)}\mathrm{d} u.
\]	
As $\Gamma$ is discrete we may assume that $U$ is small enough so that $U\cap\Gamma=\{e\}$. Putting $C=\frac{(C')^2}{\mathrm{vol}(U)}$ we obtain
\[
	\sum_{\gamma\in \Lambda}\frac{\pi_{\nu}(\gamma)\mathbbm 1_B(b)}{\Xi(\gamma)}\leq C \int_{\Lambda U}\frac{\pi_{\nu}(g)\mathbbm 1_B(b)}{\Xi(g)}{\rm d}g.
\]
\end{proof}

\subsection{General facts about metric Borel spaces}

Let $(B,d)$ be a metric space and $E\subset B$  a subset. For  $r>0$ we denote
\[
	E(r)=\{b\in B : d(b,E)<r\}
\] 
 the open $r$-neighborhood of $E$ in $B$.
Recall that the boundary $\partial E$ of a subset $E$ in a topological space $B$ is the intersection of its closure with the closure of its complement
\[
	\partial E=\overline{E}\cap\overline{B\setminus E}
\]
and that it is Borel in the case $E$ is Borel. For the sake of clarity we recall in the following proposition two facts (standard and easy to prove) we will use about metric Borel spaces.

\begin{prop}\label{proposition: measure} Let $(B,\nu)$ be a Borel space with a probability measure. We assume the topology inducing the Borel structure is defined by a distance $d$.
	Let $E\subset B$ be a Borel set. 
	\begin{enumerate}
		\item If $\nu(\partial E)=0$ then 
		$\lim_{r\to 0}\nu(E(r))=\nu(E)$.
		\item The set $\{r>0: \nu(\partial E(r))\neq 0\}$ is at most countable.
	\end{enumerate}
\end{prop}	
\subsection{Bounded operators preserving positive functions}
Let $(B,\nu)$ be a Borel space with a probability measure. We assume the topology inducing the Borel structure is defined by a distance $d$. Suppose for each Borel bounded function $f$ on $B$ we are given a bounded operator
\[
	M^f\in {\mathcal B}\bigl(L^2(B,\nu)\bigr),
\]
and the following properties hold:
\begin{enumerate}
	\item $M^{af+bg}=aM^f+bM^g\quad \forall a,b\in\mathbb C\ \forall f,g \mbox{ Borel bounded}$,
	\item $f\geq 0 \ \Rightarrow\  M^f\varphi\geq 0\,\, \forall \varphi\in L^2(B,\nu)\mbox{ with }  \varphi\geq 0$,
	\item $\langle M^{\mathbbm 1_B}\mathbbm 1_B,\mathbbm 1_B\rangle=1$.  
\end{enumerate}
The following lemma is an abstract version of \cite[Proposition 5.5]{BadMuc}. The adaptation of the proof from \cite{BadMuc} is not difficult and is left to the reader.
\begin{lemma}\label{lemma: boundary of measure zero} Let $(B,\nu)$ and $d$ be as above. Let $M_n$ be a sequence of linear transformations from the space of bounded Borel functions on $B$ to 
${\mathcal B}\bigl(L^2(B,\nu)\bigr)$ as above. Assume the following two conditions hold:
\begin{enumerate}
	\item for all Borel subsets $U,V\subset B$ with $\nu(\partial U)=\nu(\partial V)=0$
	\[
		\limsup_{n\to\infty}\langle M_n^{\mathbbm 1_U}\mathbbm 1_V,\mathbbm 1_B\rangle\leq\nu(U)\nu(V),
	\]
	\item for all Borel subsets $U,W\subset B$ such that $d(U,W)>0$
	\[
		\lim_{n\to\infty}\langle M_n^{\mathbbm 1_U}\mathbbm 1_B,\mathbbm 1_W\rangle=0.
	\]
\end{enumerate}
Then for all Borel subsets $U,V,W\subset B$ with $\nu(\partial U)=\nu(\partial V)=\nu(\partial W)=0$ we have
	\[
		\lim_{n\to\infty}\langle M_n^{\mathbbm 1_U}\mathbbm 1_V,\mathbbm 1_W\rangle=\nu(U\cap W)\nu(V).
	\]	
\end{lemma}

\subsection{Dirac type peaks and counting}
Recall that a \emph{length function} $L:G\to[0,\infty[$ on a locally compact group $G$ is a locally bounded proper Borel map such that
for all $g,h\in G$ we have  $L(gh)\leq L(g)+L(h)$, $L(g^{-1})=L(g)$ and $L(e)=0$ where $e\in G$ denotes the neutral element.
The following lemma is a generalization of \cite[Proposition 5.1]{BadMuc}. We give the proof for the convenience of the reader.

\begin{lemma}\label{lemma: Dirac} Let $(B,\nu)$ be a Borel probability space with a distance $d$ inducing the topo\-lo\-gy of $B$.
Let $\pi_{\nu}:G\to \mathcal{U}\bigl(L^2(B,\nu)\bigr)$ be a quasi-regular representation. Let $L$ be a length function on $G$.	
For each $n\in\mathbb N$ let $\mu_n\in L^1(G,{\rm d}g)$ such that $\|\mu_n\|_1\leq 1$, $\mu_n(g)\geq 0$  a.e. and
\[
	\lim_{n\to\infty}\mu_n(g)=0.
\]
Let  $\mathrm{supp}(\mu_n)$ denote the (essential) support of $\mu_n$.
Let $b_n:\mathrm{supp}(\mu_n)\to B$ be a Borel map.
Assume that for each $r>0$ there exists a function $h_r:[0,\infty[\to[0,\infty[$ with the following properties:
\begin{enumerate}
	\item $h_r$ is non-increasing,
	\item $\lim_{s\to\infty}h_r(s)=0,$
	\item $\forall n\in\mathbb N\ \forall g\in \mathrm{supp}(\mu_n)$
	\[
	\frac{\langle\pi_{\nu}(g)\mathbbm 1_B,\mathbbm 1_{\{b\in B : d(b,b_n(g))\geq r \}}\rangle}{\Xi(g)}\leq h_r(L(g)).
	\]
\end{enumerate}
Then for any Borel subset $W$ in $B$ and any $r>0$
\begin{align*}
	\limsup_{n\to\infty}\int_G\mu_n(g)\frac{\langle\pi_{\nu}(g)\mathbbm 1_B,\mathbbm 1_W\rangle}{\Xi(g)}{\rm d}g\leq
	\limsup_{n\to\infty}\int_{\mathrm{supp}(\mu_n)}\mu_n(g)\mathbbm 1_{W(r)}(b_n(g)){\rm d}g.
\end{align*}
\end{lemma}	
\begin{proof} Let $W\subset B$ be a Borel subset and let $r>0$. For each $n\in\mathbb N$ and each $s>0$ we decompose the support of $\mu_n$ as
\begin{align*}
      \mathrm{supp}(\mu_n)&=\{g\in \mathrm{supp}(\mu_n) : L(g)<s\}\\
	             &\bigcup\{g\in \mathrm{supp}(\mu_n) : L(g)\geq s,\ b_n(g)\in W(r)\}\\
				 &\bigcup\{g\in \mathrm{supp}(\mu_n) : L(g)\geq s,\ b_n(g)\notin W(r)\}.	
\end{align*}
If $g$ is such that $b_n(g)\notin W(r)$ then obviously $\mathbbm 1_W\leq\mathbbm 1_{\{b\in B : d(b,b_n(g))\geq r \}}$. Hence we have
\begin{align*}
     \int_G\mu_n(g)\frac{\langle \pi_{\nu}(g)\mathbbm 1_B,\mathbbm 1_W\rangle}{\Xi(g)}{\rm d}g&=
     \end{align*}
     \begin{align*}
	          & =\int_{\{g\in \mathrm{supp}(\mu_n) : L(g)<s\}}\mu_n(g)\frac{\langle\pi_{\nu}(g)\mathbbm 1_B,\mathbbm 1_W\rangle}{\Xi(g)}{\rm d}g \\
	          &\qquad +\int_{\{g\in \mathrm{supp}(\mu_n) : L(g)\geq s,\, b_n(g)\in W(r)\}}\mu_n(g)\frac{\langle\pi_{\nu}(g)\mathbbm 1_B,\mathbbm 1_W\rangle}{\Xi(g)}{\rm d}g\\
		&\qquad+\int_{\{g\in \mathrm{supp}(\mu_n) : L(g)\geq s,\, b_n(g)\notin W(r)\}}\mu_n(g)\frac{\langle\pi_{\nu}(g)\mathbbm 1_B,\mathbbm 1_W\rangle}{\Xi(g)}{\rm d}g\\
		 &\leq \int_{\{g\in \mathrm{supp}(\mu_n) : L(g)<s\}}\mu_n(g){\rm d}g\\
	          &\qquad+\int_{\{g\in \mathrm{supp}(\mu_n) : L(g)\geq s,\, b_n(g)\in W(r)\}}\mu_n(g){\rm d}g\\
	&\qquad+\int_{\{g\in \mathrm{supp}(\mu_n) : L(g)\geq s,\, b_n(g)\notin W(r)\}}\mu_n(g)h_r(s){\rm d}g.\\	
\end{align*}
As 	$\{g\in G: L(g)<s\}$ is relatively compact it has finite Haar measure, and Lebesgue's dominated convergence theorem implies
\[
	\lim_{n\to\infty}\int_{\{g\in G: L(g)<s\}}\mu_n(g){\rm d}g=\int_{\{g\in G: L(g)<s\}}\lim_{n\to\infty}\mu_n(g){\rm d}g=0.
\]
Hence, taking the limit superior in the above inequality we obtain
\begin{align*}
     &\limsup_{n\to\infty}\int_G\mu_n(g)\frac{\langle\pi_{\nu}(g)\mathbbm 1_B,\mathbbm 1_W\rangle}{\Xi(g)}{\rm d}g\\
	 &\leq\limsup_{n\to\infty}\int_{\mathrm{supp}(\mu_n)}\mu_n(g)\mathbbm 1_{W(r)}(b_n(g)){\rm d}g\\
	 &+h_r(s).	
\end{align*}
As the above inequality is true for any $s>0$ and as $\lim_{s\to\infty}h_r(s)=0$ the lemma is proved.

\end{proof}

\begin{prop}\label{proposition: away}(Away from the peak of the square root of the Poisson kernel.) Let $\pi_{\nu}:G\to \mathcal{U}\bigl(L^2(B,\nu)\bigr)$ be a quasi-regular representation, and let $L$ be a length function on $G$.
	For each $n$  let $E_n\subset G$ be a  relatively compact Borel subset and let $e_n:E_n\to B$ be a  Borel map. Assume
	\[
		\lim_{n\to\infty}\mathrm{vol}(E_n)=\infty.
	\]
	Assume that for each $r>0$ there exists a function $h_r:[0,\infty[\to[0,\infty[$ with the following properties:
	\begin{enumerate}
		\item $h_r$ is non-increasing,
		\item $\lim_{s\to\infty}h_r(s)=0,$
		\item $\forall n\in\mathbb N\ \forall g\in E_n$
		\[
		\frac{\langle\pi_{\nu}(g)\mathbbm 1_B,\mathbbm 1_{\{b\in B : d(b,e_n(g))\geq r \}}\rangle}{\Xi(g)}\leq h_r(L(g)).
		\]		
	\end{enumerate}
	Then for any Borel subset $U\subset B$ the operators
	\[
		M_n^{\mathbbm 1_U}=M_{(E_n,e_n)}^{\mathbbm 1_U}=\frac{1}{\mathrm{vol}(E_n)}\int_{E_n}\mathbbm 1_U(e_n(g))\frac{\pi_{\nu}(g)}{\Xi(g)}{\rm d}g
	\]
	have the following properties:
	\begin{enumerate}
		\item if $U,W\subset B$ are Borel subsets such that $d(U,W)>0$ then
		\[
			\lim_{n\to\infty}\langle M_n^{\mathbbm 1_U}\mathbbm 1_B,\mathbbm 1_W\rangle=0,
		\]
		\item if $G$ is unimodular and  $E_n=E_n^{-1}$ (i.e. stable under taking inverses in $G$) then for all Borel subsets
		$U,V\subset B$ and all $r>0$
		\begin{align*}
			&\limsup_{n\to\infty}\langle M_n^{\mathbbm 1_U}\mathbbm 1_V,\mathbbm 1_B\rangle\\
			&\leq\limsup_{n\to\infty}\frac{1}{\mathrm{vol}(E_n)}\int_{E_n}\mathbbm 1_U(e_n(g^{-1}))\mathbbm 1_{V(r)}(e_n(g)){\rm d}g.
		\end{align*}
	\end{enumerate}
\end{prop}
\begin{proof} To show the first claim we 
 suppose $d(U,V)>0$. We choose $0<r<d(U,V)$ and define $\mu_n(g)=0$ if $g\notin E_n$, and
	$$\mu_n(g)=\frac{1}{\mathrm{vol}(E_n)}\mathbbm 1_{E_n}(g)\mathbbm 1_U(e_n(g))$$
if $g\in E_n$.	
Applying Lemma \ref{lemma: Dirac} we obtain
\begin{align*}
     \limsup_{n\to\infty}\langle M_{(E_n,e_n)}^{\mathbbm 1_U}\mathbbm 1_B,\mathbbm 1_W\rangle &=\limsup_{n\to\infty}\int_G\mu_n(g)\frac{\langle\pi_{\nu}(g)\mathbbm 1_B,\mathbbm 1_W\rangle}{\Xi(g)}{\rm d}g\\
	                                                 &\leq\limsup_{n\to\infty}\int_{\mathrm{supp}(\mu_n)}\mu_n(g)\mathbbm 1_{W(r)}(e_n(g)){\rm d}g\\
													 &=\limsup_{n\to\infty}\frac{1}{\mathrm{vol}(E_n)}\int_{E_n}\mathbbm 1_U(e_n(g))\mathbbm 1_{W(r)}(e_n(g)){\rm d}g.
\end{align*}
but $\mathbbm 1_U(e_n(g))\mathbbm 1_{W(r)}(e_n(g))=0$ for all $g\in E_n$ because $U\cap W(r)=\emptyset$.	

For the second claim we assume that  $G$ is unimoduar and that $E_n=E_n^{-1}$. The adjoint of $M_{(E_n,e_n)}^{\mathbbm 1_U}$ is
\begin{align*}
	(M_{(E_n,e_n)}^{\mathbbm 1_U})^*&=\frac{1}{\mathrm{vol}(E_n)}\int_{E_n}\mathbbm 1_U(e_n(g))\frac{\pi_{\nu}^*(g)}{\Xi(g)}{\rm d}g\\
	                       &=\frac{1}{\mathrm{vol}(E_n)}\int_{E_n}\mathbbm 1_U(e_n(g))\frac{\pi_{\nu}(g^{-1})}{\Xi(g)}{\rm d}g\\
						   &=\frac{1}{\mathrm{vol}(E_n)}\int_{E_n^{-1}}\mathbbm 1_U(e_n(g^{-1}))\frac{\pi_{\nu}(g)}{\Xi(g^{-1})}{\rm d}g^{-1}\\
						   &=\frac{1}{\mathrm{vol}(E_n)}\int_{E_n}\mathbbm 1_U(e_n(g^{-1}))\frac{\pi_{\nu}(g)}{\Xi(g)}{\rm d}g.
\end{align*}
We define $\mu_n(g)=0$ if $g\notin E_n$ and
	$$\mu_n(g)=\frac{1}{\mathrm{vol}(E_n)}\mathbbm 1_{E_n}(g)\mathbbm 1_U(e_n(g^{-1}))$$
if $g\in E_n$.	
Applying Lemma \ref{lemma: Dirac} we obtain for any Borel subset $V\subset B$ and any $r>0$
\begin{align*}
     \limsup_{n\to\infty}\langle M_{(E_n,e_n)}^{\mathbbm 1_U}\mathbbm 1_V,\mathbbm 1_B\rangle&=\limsup_{n\to\infty}\langle(M_{(E_n,e_n)}^{\mathbbm 1_U})^*\mathbbm 1_B,\mathbbm 1_V\rangle\\
	                                                 &=\limsup_{n\to\infty}\int_G\mu_n(g)\frac{\langle\pi_{\nu}(g)\mathbbm 1_B,\mathbbm 1_V\rangle}{\Xi(g)}{\rm d}g\\
	                                                 &\leq\limsup_{n\to\infty}\int_{\mathrm{supp}(\mu_n)}\mu_n(g)\mathbbm 1_{V(r)}(e_n(g)){\rm d}g\\
													 &=\limsup_{n\to\infty}\frac{1}{\mathrm{vol}(E_n)}\int_{E_n}\mathbbm 1_U(e_n(g^{-1}))\mathbbm 1_{V(r)}(e_n(g)){\rm d}g.
\end{align*}
\end{proof}

\subsection{Proof of Theorem \ref{theorem: ergodicity of some quasi-regular representations}}
\begin{proof} In order to prove the first part of the theorem, namely
\[
	\lim_{n\to\infty}\langle M_{(E_n,e_n)}^{\mathbbm 1_U}\mathbbm 1_V,\mathbbm 1_W\rangle=\nu(U\cap W)\nu(V)
\]	
for Borel sets with zero measure boundary, it is enough to check that the hypotheses of Lemma~\ref{lemma: boundary of measure zero} are satisfied.
	The operator $M_{(E_n,e_n)}^{f}$ is obviously positive if $f$ is, and $\langle M_{(E_n,e_n)}^{\mathbbm 1_B}\mathbbm 1_B,\mathbbm 1_B\rangle=1$.
	
	It follows from Proposition \ref{proposition: away} that  the second assumption in Lemma~\ref{lemma: boundary of measure zero} is satisfied, namely for all Borel subsets $U,W\subset B$ with  $d(U,W)>0$ 
		\[
			\lim_{n\to\infty}\langle M_{(E_n,e_n)}^{\mathbbm 1_U}\mathbbm 1_B,\mathbbm 1_W\rangle=0.
		\]
	
	Let us check that the first assumption in Lemma  \ref{lemma: boundary of measure zero} is satisfied, namely that for all Borel subsets $U,V\subset B$ with $\nu(\partial U)=\nu(\partial V)=0$
	\[
		\limsup_{n\to\infty}\langle M_{(E_n,e_n)}^{\mathbbm 1_U}\mathbbm 1_V,\mathbbm 1_B\rangle\leq\nu(U)\nu(V).
	\]
	From Proposition \ref{proposition: away} we deduce that for all $r>0$
	\begin{align*}
		&\limsup_{n\to\infty}\langle M_{(E_n,e_n)}^{\mathbbm 1_U}\mathbbm 1_V,\mathbbm 1_B\rangle\\
		&\quad \leq\limsup_{n\to\infty}\frac{1}{\mathrm{vol}(E_n)}\int_{E_n}\mathbbm 1_U(e_n(g^{-1}))\mathbbm 1_{V(r)}(e_n(g)){\rm d}g.
	\end{align*}
	Hence, according to Proposition  \ref{proposition: measure}, for any $\epsilon>0$ we may choose $r$ such that $0<r<\epsilon$ and such that   
	$\nu(\partial(V(r)))=0$. Applying the above inequality with the chosen $r$ and the second assumption of the theorem  we obtain
	\[
		\limsup_{n\to\infty}\langle M_{(E_n,e_n)}^{\mathbbm 1_U}\mathbbm 1_V,\mathbbm 1_B\rangle\leq\nu(U)\nu(V(r)).
	\]
	As $r$ may be chosen arbitrarily small in the above inequality and as $\nu(\partial V)=0$ we deduce that
	\[
		\limsup_{n\to\infty}\langle M_{(E_n,e_n)}^{\mathbbm 1_U}\mathbbm 1_V,\mathbbm 1_B\rangle\leq\nu(U)\nu(V),
	\] 
	since according to Proposition \ref{proposition: measure} $\lim_{r\to 0}\nu(V(r))=\nu(V)$. This terminates the proof of the first part of the theorem.
	
To prove the second statement in the theorem, we notice first that linearity and  the first part of the theorem imply  that for any $h\in H$ and for all Borel sets $V,W$ in $B$ with $\nu(\partial V)=\nu(\partial W)=0$  we have
\begin{equation}\label{equation : characteristic functions}
	\lim_{n\to\infty}\langle M_{(E_n,e_n)}^{h}\mathbbm 1_V,\mathbbm 1_W\rangle=\langle\mathbbm 1_V,\mathbbm 1_B\rangle\langle h,\mathbbm 1_W\rangle.
\end{equation}
Applying Lemma \ref{lemma: averaging the cocycle at the point b} it is easy to  deduce that 
(\ref{equation : characteristic functions}) also holds for any $h\in \overline H$.

To finish the proof of the theorem we check that Lemma~\ref{lemma: straightforward} applies. 

The density hypothesis is satisfied: the topology on $B$ is induced by a metric  hence the probability Borel measure $\nu$ is regular, it implies that the smallest subspace of $L^2(B,\nu)$ 
containing the characteristic functions of Borel subsets of $B$ whose boundaries have zero measure is dense. 

The uniform boundedness hypothesis is satisfied thanks to Lemma \ref{lemma: averaging the cocycle at the point b}
\end{proof}

\section{Furstenberg-Poisson boundaries}
Let $G$ be a non-compact semisimple connected Lie group with finite center. We use the notation introduced in Subsection \ref{subsection: statement of the main results}.  

\subsection{Cones around the barycenter of a Weyl chamber}
Let 
\[
	\rho=\frac{1}{2}\sum_{\alpha\in\Sigma^+}\dim\frak g_{\alpha}\cdot\alpha
\]
be half the sum of the positive roots counted with multiplicity.
The maximum of $\rho$ on $\frak a^+\cap \overline{\frak a_1}$ is reached at a unique vector  $H_{max}$ of norm $1$ called the barycenter
of $\frak a^+$.
Let $\theta>0$. We define the open cone in $\frak a^+$ of angle $2\theta$ around $H_{max}$ as
\[
	\frak a^{\theta}=\{H\in\frak a^+: \measuredangle (H, H_{max})<\theta\},
\]
we denote the intersection (truncation) of this cone with the ball of radius $T$ and center the origin in $\frak a$ as
\[
	\frak a_T^{\theta}=\frak a^{\theta}\cap \frak a_T
\]
and  its image in $G$ under the exponential map as
\[
	A_T^{\theta}=\exp(\frak a_T^{\theta}).
\]
Eventually, we define 
\[
	B_T^{\theta}=KA_T^{\theta}K.
\]
\subsection{Volume estimates}
We normalize the  Lebesgue measure ${\rm d}H$ on $\frak{a}$ and the Haar measure ${\rm d}g$ on $G$ so that
for any continuous function $f$ on $G$ with compact support  we have
\[
\int_G f(g){\rm d}g=\int_K \int_{\frak{a}^+} \int_K f(k\exp(H)l)J(H) {\rm d}k {\rm d}H {\rm d}l,
\]
where \[ J(H) =\prod_{\alpha\in\Sigma^+}(\sinh\alpha(H))^{\dim\frak{g}_\alpha}\quad\text{for }\ H\in\frak{a}^+.
\]
Let $E$ be a Borel subset of $G$. We denotes its volume as
\[
	\mathrm{vol}(E)=\int_G \mathbbm 1_E(g){\rm d}g.
\] 
In particular, for any measurable subset $L$ of $\frak a^+$ we have, for $D=\exp L$,
\[ 
\mathrm{vol}\left(K D K\right)= \int_{L}J(H){\rm d}H. 
\]
The growth rate $\delta$ of the symmetric space $(X,d_X)$ is given by the formula
\[
	\delta=2\rho(H_{max}).
\] 
(See \cite[5. Volume estimates]{GorOh}.)
It is sometimes convenient to normalize the metrics $d_G$ and $d_X$  defined in Subsection \ref{subsection: statement of the main results}. In the case the real rank 
$$r=\mathrm{rank}_{\mathbb R}G=\dim A$$
of $G$ is one, the normalization condition
\[
	d_X(\exp(H)x_0,x_0)=\beta(H)\quad \forall H\in\frak a^+,
\]
where $\beta$ is the unique positive indivisible root, leads to  a definition of $\delta$ in terms of the root system only: 
\[
	2\rho=\delta\beta.
\]
The values of $\delta$  are $n-1$ for $\mathrm{SO}(n,1)$ and $2n$ for $\mathrm{SU}(n,1)$. See \cite[2.8 Lemma]{Sha}. In the case $X=\mathbb H^n$ is the hyperbolic $n$-space, the normalized metric $d_{\mathbb H^n}$ has constant curvature equal to $-1$.

Lemma 5.4 in \cite{GorOh} states that for the sets $A^\theta_T$ we have 
\[ 
\mathrm{vol}\left(K A^\theta_T K\right)\sim c T^{(r-1)/2}\cdot e^{\delta T}
\]
with a constant $c>0$  \emph{independent} of the choice of $\theta$.  As usual, the notation  $f\sim g$  for two  functions $f,g:[0,\infty[\to\mathbb{R}$ means that $\lim_{T\to\infty}f(T)/g(T)=1$.
For $$G_T=K\overline{A^+_T}K$$ we also have
\[ 
\mathrm{vol}\left(G_T\right)\sim c T^{(r-1)/2}\cdot e^{\delta T}
\]
with the same constant $c$ as above.  

\subsection{Ergodicity relative to cones}

If $\psi\in L^{\infty}(G/P,\nu)$, we denote $m(\psi)\in {\mathcal B}\bigl(L^{2}(G/P,\nu)\bigr)$ the corresponding multiplication operator.
Let $P_{\mathbbm 1_{G/P}}\in {\mathcal B}\bigl(L^{2}(G/P,\nu)\bigr)$ be the orthogonal projection onto the complex line spanned by $\mathbbm 1_{G/P}$.
Let $\Gamma$ be a discrete subgroup of $G$,
\[
	\Gamma_T^{\theta}=B_T^{\theta}\cap\Gamma
\]
and let $|\Gamma_T^{\theta}|$ denote the cardinality of this finite set. 
Let $f:G/P\to\mathbb C$ be a continuous function. In the case  $\Gamma_T^{\theta}$ is non-empty we may consider 
the bounded operator
\[
	M_{\Gamma_T^{\theta}}^f=\frac{1}{|\Gamma_T^{\theta}|}\sum_{\gamma\in \Gamma_T^{\theta}}f(\bold {b}(\gamma))\frac{\lambda_{G/P}(\gamma)}{\Xi(\gamma)}.
\]  

\begin{theo}\label{theo: ergodicity relative to cones of the quasi-regular representation of a lattice in a semisimple Lie group}(Ergodicity  of the quasi-regular representation of a lattice in a semisimple Lie group, relative to cones.) Let $G$ be a non-compact connected semisimple Lie group with finite center. Let $P$ be a minimal parabolic subgroup of $G$ and  $\Gamma$  a lattice in $G$. Let $\theta>0$ and let $f$ be a continuous function on $G/P$. With the notation as above we have
    \[
		\lim_{T\to\infty}M_{\Gamma_T^{\theta}}^f=m(f)P_{\mathbbm 1_{G/P}}
	\]
in the weak operator topology of ${\mathcal B}\bigl(L^{2}(G/P,\nu)\bigr)$.	
That is, for any $\varphi,\psi\in L^{2}(G/P,\nu)$
	\[
		\lim_{T\to\infty}
		\frac{1}{|\Gamma_T^{\theta}|}\sum_{\gamma\in\Gamma_T^{\theta}}
		f(\bold {b}(\gamma))\frac{\langle\lambda_{G/P}(\gamma)\varphi,\psi\rangle}{\Xi(\gamma)}=
		\langle\varphi,\mathbbm 1_{G/P}\rangle\langle f,\psi\rangle.
	\]
\end{theo}

\subsection{Convergence for the square root of the Poisson kernel}
\label{subsection: Sjogren}
It is convenient to identify the Furstenberg-Poisson boundary $B=G/P$ with the space $K/M$ using the diffeomorphism
\[
	K/M\to G/P
\]
which sends $kM$ to $kP$. We consider the  unique action of $G$ on $K/M$ which makes 
this diffeomorphism $G$-equivariant.  We  denote $\nu$ the push-forward on $K/M$ of the probability Haar measure on $K$. 
Let $G=KAN$ be the Iwasawa decomposition defined by $\frak a^+$. If $g=kan$ we denote its $\frak a$-component by
$H_I(g)=\log(a)$. Notice that if $m\in M$ then $H_I(gm)=H_I(g)$. 
The Radon-Nikodym cocycle for the action of $g\in G$ at the point $kM\in K/M$ is
$$
c(g,kM)=e^{-2\rho(H_I(gk))}.
$$
See for example \cite[Proposition 2.5.4]{GanVar}. The quasi-regular representation of $G$ on 
$\varphi\in L^2(K/M,\nu)$ is defined as
$$
(\pi_{\nu}(g)\varphi)(kM)=\varphi(g^{-1}kM)e^{-\rho(H_I(g^{-1}k))}.
$$
See for example \cite[(3.1.12) page 103]{GanVar}. The Harish-Chandra function is
\begin{align*}
\Xi(g)&=\langle\pi_{\nu}(g)\mathbbm 1_{K/M},\mathbbm 1_{K/M}\rangle=\int_{K/M}e^{-\rho(H_I(g^{-1}k))}{\rm d}\nu(kM)\\
      &=\int_K e^{-\rho(H_I(g^{-1}k))}{\rm d}k.
\end{align*}
Let $\varphi\in L^1(K/M,\nu)$.
The normalized square root of the Poisson kernel is
\begin{align*}
(\mathcal{P}_0\varphi)(g)&=\frac{\langle\pi_{\nu}(g)\mathbbm 1_{K/M},\overline{\varphi}\rangle}{\Xi(g)}\\
                     &=\frac{\int_{K/M} \varphi(kM)e^{-\rho(H_I(g^{-1}k))}{\rm d}\nu(kM)}{\Xi(g)}.
\end{align*}
See \cite{Sjö}.
Let $d$ be a left $K$-invariant Riemannian distance on $K/M$. The formula $L(g)=d_X(gx_0,x_0)$ defines a left $G$-invariant and $K$-bi-invariant length metric on $G$.

\begin{lemma}\label{lemma: upper bound on the square root of the Poisson kernel in symmetric spaces}(Dirac sequences on Poisson-Furstenberg boundaries.) With the notation as above let $\pi_{\nu}: G\to \mathcal{U}\bigl(L^2(K/M,\nu)\bigr)$ be the quasi-regular representation of $G$. Let $\theta>0$ be small enough so that the intersection of $\overline{\frak{a}^{\theta}}$ with the walls of $\overline{\frak a^+}$ is reduced to the origin. 
For each $r>0$ there exists a function $h_r:[0,\infty[\to[0,\infty[$ with the following properties:
\begin{enumerate}
	\item $h_r$ is non-increasing,
	\item $\lim_{s\to\infty}h_r(s)=0,$
	\item $\forall g\in KA^{\theta}K,$
\[
	\mathcal{P}_0\mathbbm 1_{\{x\in K/M :\, d(x,\bold{b}(g)\geq r\}}(g)\leq h_r(L(g)).
\]
\end{enumerate}
\end{lemma}
\begin{proof} 
As the space $K/M$ is normal, Urysohn's Lemma applies so there exists a continuous function $f_r$ on $K/M$ with the following properties:
\begin{itemize}
	\item $f_r\geq 0,$
	\item $f_r(x)=1, \, \forall x: d(x,eM)\geq r,$
	\item $f_r(x)=0, \, \forall x: d(x,eM)\leq r/2.$
\end{itemize}
According to \cite[Theorem 5.1 p. 373]{Sjö} for any  $r>0$ and  $s\geq 0$ the supremum
$$h_r(s)=
\sup_{H\in\frak{a}^{\theta}, \|H\|\geq s}
\mathcal{P}_0f_r(\exp(H))$$
is finite and
$$\lim_{s\to\infty}h_r(s)=f_r(eM)=0.$$
The notation $H\to\infty$ in \cite{Sjö} means $\alpha(H)\to\infty$ for any positive root $\alpha$;
sequences contained in the cone $\frak{a}^{\theta}$ with $\|H\|\to\infty$ obviously verify these conditions because
by hypothesis $\theta$ is small enough so that the intersection of the closed cone $\overline{\frak{a}^{\theta}}$ with the walls of $\overline{\frak a^+}$ is reduced to the origin.

Let $g\in KA^{\theta}K$. Let $k\in K,\, H\in\frak{a}^{\theta}, l\in K$ such that $g=k\exp(H)l$.
Hence ${\bold b}(g)=kM$. As the action of $K$ on $K/M$ preserves the measure and as $d$ is $K$-invariant, we have
$$
\pi_{\nu}(k^{-1})\mathbbm 1_{\{x\in K/M :\, d(x,kM)\geq r\}}=\mathbbm 1_{\{x\in K/M :\, d(x,eM)\geq r\}}.
$$
It follows that
\begin{align*}
\mathcal{P}_0\mathbbm 1_{\{x\in K/M :\, d(x,\bold{b}(g)\geq r\}}(g)&=\frac{\langle\pi_{\nu}(k\exp(H)l)\mathbbm 1_{K/M},\mathbbm 1_{\{x\in K/M :\, d(x,kM)\geq r\}}\rangle}{\Xi(k\exp(H)l)}\\
&=\frac{\langle\pi_{\nu}(\exp(H))\mathbbm 1_{K/M},\mathbbm 1_{\{x\in K/M :\, d(x,eM)\geq r\}}\rangle}{\Xi(\exp(H))}\\
&=\mathcal{P}_0\mathbbm 1_{\{x\in K/M :\, d(x,eM)\geq r\}}(\exp(H)).
\end{align*}
As $L$ is $K$-bi-invariant we conclude that it is enough to prove the lemma in the special case $g=\exp(H)$ with $H\in\frak{a}^{\theta}$. But $\mathcal{P}_0$ preserves positive functions and $$\mathbbm 1_{\{x\in K/M :\, d(x,eM)\geq r\}}\leq f_r.$$ This finishes the proof of the lemma.

\end{proof}
\subsection{Counting lattice points and the wave-front lemma}\label{subsection: counting lattice points and the wave-front lemma}
Recall  the notation $f\sim g$ for  $\lim_{T\to\infty}f(T)/g(T)=1.$

\begin{prop}\label{proposition: counting lattice points in sectors}(Counting lattice points in sectors.)
	Let $\theta>0$. As $T\to\infty$ we have
	\[
		|\Gamma^{\theta}_T|\sim|\Gamma_T|\sim|\Gamma\cap G_T|.
	\]	
\end{prop}
\begin{proof} Notice that if $0<\theta\leq\phi$ then for all $T>0$
\[
	\Gamma^{\theta}_T\subset\Gamma^{\phi}_T\subset\Gamma_T\subset\Gamma\cap G_T.
\]
Hence, in proving the proposition, we may assume $\theta$ is small. For all $T>0$
\[
	\frac{|(\Gamma\cap G_T)\setminus \Gamma^{\theta}_T|+|\Gamma^{\theta}_T|}{|\Gamma\cap G_T|}=1.
\]
Hence the proposition will be proved if we show that
\[
	\lim_{T\to\infty}\frac{|(\Gamma\cap G_T)\setminus \Gamma^{\theta}_T|}{|\Gamma\cap G_T|}=0.
\]
To that end we introduce the following notation. For any  $\phi>0$ and $R>0$ we define
\[
	\frak{a}^{\phi}(R)=\{H\in\frak{a}^{\phi}: \|H\|\geq R\}=\frak{a}^{\phi}\setminus\frak{a}_R.
\] 
Let us choose $T_0>1/\sin(\theta/2)$. Elementary trigonometry shows that
\[
	\frak{a}^{\theta/2}(T_0+1)+\frak{a}_1\subset\frak{a}^{\theta}(T_0). 
\]
As $\theta$ is small,  the closure of $\frak{a}^{\theta/2}(T_0+1)$ in $\frak a$ lies at positive distance of the walls of the Weyl chamber $\frak{a}^+$.
Hence we can apply the strong wave-front lemma from \cite[Theorem 2.1]{GorOh} to the (closure) of the subset
\[
	A^{\theta/2}(T_0+1)=\exp(\frak{a}^{\theta/2}(T_0+1))
\]
of $\overline{A^+}$ (warning: the closed set $\overline{A^+}=\exp(\overline{\frak{a}^+})$ is denoted  $A^+$ in \cite{GorOh}) and to the neighborhood  $V=\exp(\frak{a}_1)$ of $e$ in $A$. The conclusion of Theorem 2.1 from \cite{GorOh} is the existence of a neighborhood $\mathcal O$
of $e$ in $G$ with the following property:  any $g=k\exp(H)l$ with $k,l\in K$ and $H\in A^{\theta/2}(T_0+1)$ satisfies
\[
	g{\mathcal O}^{-1}\subset K\exp(H)VK.
\]
In other words
\begin{equation}\label{equation: the strong wave-front lemma}
	KA^{\theta/2}(T_0+1)K{\mathcal O}^{-1}\subset KA^{\theta/2}(T_0+1)VK.
\end{equation}
Choosing $\mathcal O$ smaller if needed, we may assume that it satisfies the following additional properties:
\[
\Gamma\cap\left(\mathcal O\cdot{\mathcal O}^{-1}\right)=\{e\}\quad\mbox{and}\,\, G_T\mathcal O\subset G_{T+1} \  \forall T>0.
\]
Combining Inclusion~(\ref{equation: the strong wave-front lemma}) with the following one
\[
A^{\theta/2}(T_0+1)V=\exp(\frak{a}^{\theta/2}(T_0+1)+\frak{a}_1)\subset\exp(\frak{a}^{\theta}(T_0))=A^{\theta}(T_0)
\]
we deduce that
\[
	KA^{\theta/2}(T_0+1)K{\mathcal O}^{-1}\subset KA^{\theta}(T_0)K.
\]
This in turn implies
\[
	\left((\Gamma\cap G_T)\setminus KA^{\theta}(T_0)K\right)\mathcal O\subset G_{T+1}\setminus KA^{\theta/2}(T_0+1)K.
\]
As $\Gamma\cap \left(\mathcal O\cdot{\mathcal O}^{-1}\right)=\{e\}$ the union
$\bigcup_{\gamma\in\Gamma}\gamma\mathcal O$
is disjoint, hence
\begin{align*}
	|(\Gamma\cap G_T)\setminus \Gamma^{\theta}_T|&\leq |(\Gamma\cap G_T)\setminus KA^{\theta}(T_0)K|\\
	&=\frac{1}{\mathrm{vol}(\mathcal O)}\mathrm{vol}\Bigl(\bigcup_{\gamma\in (\Gamma\cap G_T)\setminus KA^{\theta}(T_0)K}\gamma\mathcal O\Bigr)\\
	&\leq \frac{1}{\mathrm{vol}(\mathcal O)}\mathrm{vol}\left(G_{T+1}\setminus KA^{\theta/2}(T_0+1)K\right)\\
	&\leq \frac{1}{\mathrm{vol}(\mathcal O)}(\mathrm{vol}\left(G_{T+1}\setminus KA^{\theta/2}K\right)+\mathrm{vol}(G_{T_0+1})).
\end{align*}
This finishes the proof of the proposition because
according to \cite[Lemma 5.4]{GorOh}, for all $\phi>0$
\[
	\lim_{T\to\infty}\frac{\mathrm{vol}(G_T\setminus KA^{\phi}K)}{\mathrm{vol}(G_T)}=0,
\] 
and because there exists $C>0$ such that for all $T>1$
\begin{equation}\label{equation: vol(G_{T+1})}
	\mathrm{vol}(G_{T+1})\leq C \mathrm{vol}(G_T),
\end{equation}
and since according to \cite{EskMcM}
\begin{equation}\label{equation: EskMcM}
	|\Gamma\cap G_T|\cdot \mathrm{vol}(G/\Gamma)\sim \mathrm{vol}(G_T)
\end{equation}	
as $T\to\infty$.
\end{proof}

\begin{lemma}\label{lemma: half Roblin} Let $U,V$ be Borel subsets of $K/M$ with $\nu(\partial U)=\nu(\partial V)=0$.
Let $\theta>0$. Then
\[
	\limsup_{T\to\infty}\frac{|\{\gamma\in\Gamma^{\theta}_T: \bold{b}(\gamma)\in U\,\, \mbox{and}\,\, \bold{b}(\gamma^{-1})\in V \}|}{|\Gamma^{\theta}_T|}\leq\nu(U)\nu(V).
\]	
\end{lemma}
\begin{proof}
Let $\mu$ be the probability Haar measure on $K$. Let $p:K\to K/M$ be the canonical projection. We have $p_*\mu=\nu$. Let
$\tilde U=p^{-1}(U)$ and $\tilde V=p^{-1}(V)$. 
Let $N(A)$ be the normalizer of $A$ in $G$ and let $M'=N(A)\cap K$. Recall that the Weyl group is the quotient $W=M'/M$ and that it contains a unique element $s_0$ which sends $\frak{a}^{+}$ to $-\frak{a}^{+}$.
Let $m_0\in M'$ such that $m_0M=s_0$. 

We claim that for all $T>0$
\begin{equation}\label{equation: inclusion}
\{\gamma\in\Gamma^{\theta}_T: \bold{b}(\gamma)\in U\,\, \mbox{and}\,\, \bold{b}(\gamma^{-1})\in V \}\subset
\Gamma\cap\tilde U\overline{A^+_T}m_0^{-1}(\tilde V)^{-1}.	
\end{equation}
Let $\gamma\in\Gamma^{\theta}_T$ such that $\bold{b}(\gamma)\in U$ and $\bold{b}(\gamma^{-1})\in V$. According to the Cartan decomposition there exist $k,l\in K$ and $H\in\frak{a}^+_T$ such that $\gamma=k\exp(H)l$. The proof will be complete if we show that
$k\in\tilde U$ and $l\in m_0^{-1}(\tilde V)^{-1}$. By definition $\bold{b}(\gamma)=kM=p(k)$ and by hypothesis $\bold{b}(\gamma)\in U$.
Hence $k\in\tilde U$. Recall that the \emph{opposition involution}
\[
	\iota:\frak{a}^+\to\frak{a}^+,\quad H\mapsto -\mathrm{Ad}(m_0)H
\]
satisfies $\exp(-H)=m_0^{-1}\exp(\iota H)m_0$  for all $H\in\frak{a}^+$.
We have
\[
	\gamma^{-1}=l^{-1}\exp(-H)k^{-1}=l^{-1}m_0^{-1}\exp(\iota H)m_0k^{-1}.
\]
As $l^{-1}m_0^{-1}\in K$,  $\iota H\in\frak{a}^+$ and  $m_0k^{-1}\in K$, this proves that
$\bold{b}(\gamma^{-1})=l^{-1}m_0^{-1}M=p(l^{-1}m_0^{-1})$. By hypothesis $\bold{b}(\gamma^{-1})\in V$ hence $l^{-1}m_0^{-1}\in\tilde V$ and
$l\in m_0^{-1}(\tilde V)^{-1}$. This finishes the proof of the claim. 

We apply \cite[Theorem 1.6]{GorOh} with $\Omega_1=\tilde U$ and $\Omega_2=m_0^{-1}(\tilde V)^{-1}$. Notice that for $i=1,2$, $\mu(\partial\Omega_i)=0$ because $\nu(\partial U)=\nu(\partial V)=0$.
Since $\mu(\Omega_1 M)=\mu(\tilde U M)=\mu(\tilde U)=\nu(U)$,  
\begin{align*}
\mu(M\Omega_2)&=\mu(Mm_0^{-1}(\tilde V)^{-1})=\mu(m_0Mm_0^{-1}(\tilde V)^{-1})\\
              &=\mu(M(\tilde V)^{-1})=\mu(\tilde V M)=\mu(\tilde V)=\nu(V)
\end{align*}
and
\begin{align*}
	\tilde U\overline{A^+_T}m_0^{-1}(\tilde V)^{-1}&=\tilde UM\overline{A^+_T}m_0^{-1}(\tilde V)^{-1}
	                                               =\tilde U\overline{A^+_T}Mm_0^{-1}(\tilde V)^{-1}\\
												   &=\Omega_1\overline{A^+_T}M\Omega_2,
\end{align*}

we obtain
\[
	|\Gamma\cap \tilde U\overline{A^+_T}m_0^{-1}(\tilde V)^{-1}|\mathrm{vol}(G/\Gamma)\sim\nu(U)\nu(V)\mathrm{vol}(G_T)
\]
as $T\to\infty$. 
Applying 
the equivalence 

\[
	|\Gamma\cap G_T|\cdot \mathrm{vol}(G/\Gamma)\sim \mathrm{vol}(G_T)
\]
as $T\to\infty$, Proposition~\ref{proposition: counting lattice points in sectors} and Inclusion~(\ref{equation: inclusion}) 
finishes the proof of the lemma.
\end{proof} 

\subsection{Uniformly bounded family of Markov operators}

\begin{prop}\label{proposition: K-invariance} Let $E\subset G$ be a Borel subset of finite non-zero Haar measure. Suppose $KE=E$.
	Then $$M_E^{\mathbbm 1_{K/M}}\mathbbm 1_{K/M}(x)=1\quad  \forall x\in K/M.$$
\end{prop}
\begin{proof} As $E$ is $K$-invariant, a change of variables shows that the function is constant on any orbit of $K$. As the action of $K$ is transitive the function is constant. Applying Fubini shows that its integral on the probability space $(K/M,\nu)$ equals $1$.
\end{proof}

\begin{prop}\label{proposition: uniform boundedness} For any $\theta>0$
\[
	\sup_{T>1}\|M_{\Gamma_T^{\theta}}^{\mathbbm 1_{K/M}}\mathbbm 1_{K/M}\|_{\infty}<\infty.
\]	
\end{prop} 
\begin{proof}
Let $U$ be a neighborhood of $e$ in $G$ and $C>0$ as in Lemma~\ref{lemma: from discrete to continuous}. We may choose $U$ small enough so that for all $T>0$ we have $\Gamma_T^{\theta}U\subset G_{T+1}$. For any  $x\in K/M$ we have
\begin{align*}
	M_{\Gamma_T^{\theta}}^{\mathbbm 1_{K/M}}\mathbbm 1_{K/M}(x)&=\frac{1}{|\Gamma_T^{\theta}|}\sum_{\gamma\in \Gamma_T^{\theta}}\frac{\pi_{\nu}(\gamma)\mathbbm 1_{K/M}(x)}{\Xi(\gamma)}\\
	&\leq \frac{C}{|\Gamma_T^{\theta}|} \int_{\Gamma_T^{\theta}U}\frac{\pi_{\nu}(g)\mathbbm 1_{K/M}(x)}{\Xi(g)}{\rm d}g\\
	&\leq \frac{C}{|\Gamma_T^{\theta}|} \int_{G_{T+1}}\frac{\pi_{\nu}(g)\mathbbm 1_{K/M}(x)}{\Xi(g)}{\rm d}g.
\end{align*} 
Applying Inequality~(\ref{equation: vol(G_{T+1})}) and Equivalence~(\ref{equation: EskMcM}) as well as Proposition~\ref{proposition: counting lattice points in sectors} we see that this integral is bounded above by 
\[
	C'\frac{1}{\mathrm{vol}(G_{T+1})} \int_{G_{T+1}}\frac{\pi_{\nu}(g)\mathbbm 1_{K/M}(x)}{\Xi(g)}{\rm d}g,
\]
where $C'>0$ is a constant which neither depends  on $T>1$ nor on $x\in K/M$.
Since $KG_T=G_T$ for any $T>0$, Proposition \ref{proposition: K-invariance} applies and we obtain
\[
	\frac{1}{\mathrm{vol}(G_{T+1})} \int_{G_{T+1}}\frac{\pi_{\nu}(g)\mathbbm 1_{K/M}(x)}{\Xi(g)}{\rm d}g=1.
\]
\end{proof}

\subsection{Proof of Theorem \ref{theo: ergodicity relative to cones of the quasi-regular representation of a lattice in a semisimple Lie group} in the case $\theta$ is small} 
\begin{proof} We assume $\theta$ is small enough so that the intersection of $\overline{\frak{a}^{\theta}}$ with the walls of $\overline{\frak a^+}$ is reduced to the origin. We deduce Theorem~\ref{theo: ergodicity relative to cones of the quasi-regular representation of a lattice in a semisimple Lie group} from  Theorem \ref{theorem: ergodicity of some quasi-regular representations}. Let us check that all the  hypotheses of Theorem \ref{theorem: ergodicity of some quasi-regular representations} are satisfied in the case of a lattice in a non-compact connected semisimple Lie group with finite center. As $\theta$ is small we can apply Lemma~\ref{lemma: upper bound on the square root of the Poisson kernel in symmetric spaces}. Hence the first condition is satisfied. Applying Lemma~\ref{lemma: half Roblin} we see that the second condition is satisfied. Proposition \ref{proposition: uniform boundedness} implies that 
\[
	\sup_{T>1}\|M_{\Gamma_T^{\theta}}^{\mathbbm 1_{K/M}}\mathbbm 1_{K/M}\|_{\infty}<\infty.
\]		
As $K/M$ is a compact differentiable mani\-fold (and $\nu$ is equivalent to the Lebesgue measure), characteristic functions of the type $\mathbbm 1_W$, with $W\subset K/M$ Borel such that $\nu(\partial W)=0$ span a subset of $L^{\infty}(K/M,\nu)$ whose closure contains the continuous functions on $K/M$.
\end{proof}	
\subsection{Proofs of Theorem \ref{theo: ergodicity relative to cones of the quasi-regular representation of a lattice in a semisimple Lie group} and Theorem \ref{theo: ergodicity of the quasi-regular representation of a lattice in a semisimple Lie group}} 
Notice that for $\phi$ large enough we have $\frak{a}^{\phi}=\frak{a}^+$, and therefore
\[
	\Gamma_T^{\phi}=B_T^{\phi}\cap\Gamma=KA_T^{\phi}K\cap\Gamma=KA_TK\cap\Gamma=\Gamma_T.
\]
Hence Theorem \ref{theo: ergodicity of the quasi-regular representation of a lattice in a semisimple Lie group} is a special case of Theorem \ref{theo: ergodicity relative to cones of the quasi-regular representation of a lattice in a semisimple Lie group}. We now prove Theorem \ref{theo: ergodicity relative to cones of the quasi-regular representation of a lattice in a semisimple Lie group}.
\begin{proof}
Applying Proposition \ref{proposition: uniform boundedness} and Lemma \ref{lemma: averaging the cocycle at the point b} we deduce that for any $\phi>0$ and any continuous function $f$ on $K/M$, 
\[
\sup_{T>1}\|M_{\Gamma_T^{\phi}}^f\|_{op}<\infty.	
\]
Suppose $\phi>0$ is given. Let $0<\theta<\phi$ be small enough so that we can apply the already proven special case of Theorem \ref{theo: ergodicity relative to cones of the quasi-regular representation of a lattice in a semisimple Lie group}. Hence, for any continuous function $f$ on $K/M$ the quasi-regular representation $\pi_{\nu}$ is ergodic relative to $(\Gamma_T^{\theta},{\bold b}|_{\Gamma_T^{\theta}})$. It follows from Proposition~\ref{proposition: counting lattice points in sectors} that
\[
		\lim_{T\to\infty}\frac{|\Gamma^{\phi}_T\setminus \Gamma^{\theta}_T|}{|\Gamma^{\phi}_T|}=0.
\]
Hence we can apply Proposition \ref{prop: ergodicity relative to cones and subcones}. We deduce that for any continuous function $f$ on $K/M$ the quasi-regular representation $\pi_{\nu}$ is ergodic relative to $(\Gamma_T^{\phi},{\bold b}|_{\Gamma_T^{\phi}})$.
\end{proof}

\end{document}